\documentclass[a4paper,10pt]{amsart}

\usepackage[utf8]{inputenc}
\usepackage{color}
\usepackage{amssymb,esint,graphicx}
\usepackage{amsmath}
\usepackage{verbatim}
\usepackage{hyperref}
\usepackage{enumerate}

\newtheorem{theorem}{Theorem}[section]

\newtheorem{lemma}[theorem]{Lemma}
\newtheorem{proposition}[theorem]{Proposition}
\newtheorem{corollary}[theorem]{Corollary}
\newtheorem{definition}{Definition}[section]

\theoremstyle{remark}
\newtheorem{remark}[theorem]{Remark}
\theoremstyle{definition}

\numberwithin{equation}{section}

\newcommand{\R}{\ensuremath{\mathbb{R}}}

\newcommand{\Levy}{\ensuremath{\mathcal{L}}}

\newcommand{\alp}{\alpha}

\newcommand{\Div}{\mbox{div}}

\newcommand{\dd}{\,\mathrm{d}}
\newcommand{\diver}{\mathrm{div}}
\newcommand{\dell}{\partial}
\newcommand{\indikator}{\mathbf{1}_{|z|\leq 1}}
\newcommand{\veps}{\varepsilon}
\newcommand{\qqquad}{\qquad\quad}
\newcommand{\e}{\text{e}}

\DeclareMathOperator*{\esssup}{ess \, sup}

\DeclareMathOperator{\sgn}{\textup{sign}}

\begin{document}

\title[$L^1$ contraction for degenerate parabolic equations]{$L^1$ contraction
for bounded (non-integrable) solutions of degenerate parabolic equations}


\author[J. Endal]{J\o rgen Endal}
\address[J. Endal]{Department of Mathematical Sciences\\
Norwegian University of Science and Technology (NTNU)\\
N-7491 Trondheim, Norway} 
\email[]{jorgen.endal\@@{}math.ntnu.no}
\urladdr{http://www.math.ntnu.no/ansatte/jorgeen}

\author[E. ~R.~Jakobsen]{Espen R. Jakobsen}
\address[E. R. Jakobsen]{Department of Mathematical Sciences\\
Norwegian University of Science and Technology (NTNU)\\
N-7491 Trondheim, Norway} 
\email[]{erj\@@{}math.ntnu.no}
\urladdr{http://www.math.ntnu.no/\~{}erj/}




\begin{abstract}
We obtain new $L^1$ contraction results for bounded entropy solutions
of Cauchy problems for degenerate parabolic equations. The equations we consider
have possibly strongly degenerate local or non-local
diffusion terms. As opposed to
previous results, our results apply without any integrability
assumption on the 
solutions. They take the form of partial Duhamel 
formulas and can be seen as quantitative extensions of finite speed of
propagation local $L^1$ contraction results for scalar conservation
laws. A key ingredient in the proofs is a new and non-trivial
construction of a subsolution of a fully non-linear (dual)
equation. Consequences of our results are maximum and comparison
principles, new a priori estimates, and in the non-local case, new 
existence and uniqueness results. 
\end{abstract}

\maketitle


\section{Introduction} 
In this paper, we consider the following Cauchy problem:
\begin{equation}
\begin{cases}
 u_t + \diver\,
f(u)-\mathfrak{L}\,\varphi(u)=g(x,t) &\quad \text{in}\qquad Q_T:=\mathbb{R}^d\times(0,T),\\[0.2cm]
u(x,0)=u_0(x) & \quad \text{on}\qquad \mathbb{R}^d,
\end{cases}
\label{E}
\end{equation}
where $u=u(x,t)$ is the solution, $T>0$, $\diver$ is the
$x$-divergence. The operator $\mathfrak{L}$ will either be the
$x$-Laplacian $\Delta$, or a non-local operator  $\mathcal{L}^\mu$
defined on $C_c^\infty(\mathbb{R}^d)$ as
\begin{equation}
\mathcal{L}^\mu[\phi](x):=\int_{\mathbb{R}^d\setminus\{0\}}\phi(x+z)-\phi(x)-z\cdot D\phi(x)\mathbf{1}_{|z|\leq1}\dd \mu(z),
\label{defgenerator}
\end{equation}
where $\mu$ is a positive Radon measure,
$D$ the $x$-gradient, and $\mathbf{1}_{|z|\leq1}$ the characteristic
function of $|z|\leq1$. Throughout the paper we assume that:
\begin{align}
&f=(f_1,f_2,\ldots,f_d)\in
W_\textup{loc}^{1,\infty}(\mathbb{R},\mathbb{R}^d);
\tag{$\textup{A}_f$}
\label{fassumption}\\
&\varphi\in W_\textup{loc}^{1,\infty}(\mathbb{R}) \text{ and $\varphi$ is non-decreasing }
(\varphi'\geq0) 
;\hspace{3.5cm}
\tag{$\textup{A}_\varphi$}
\label{Aassumption}\\
&g \text{ is measurable and }\int_0^T\|g(\cdot,t)\|_{L^\infty(\mathbb{R}^d)}\dd t<\infty;
\tag{$\textup{A}_g$}
\label{gassumption}\\
&u_0\in L^{\infty}(\mathbb{R}^d);
\tag{$\textup{A}_{u_0}$}
\label{u_0assumption}\\
&\mu\geq0 \text{ is a Radon measure on }\R^d\setminus\{0\},\text{ and
  there is $M\geq0$ such that}\!\hspace{0cm}
\label{muassumption1}\tag{$\textup{A}_{\mu}$}\\ 
&\quad\int_{|z|\leq1}|z|^2\dd \mu(z)+\int_{|z|>1}\e^{M|z|}\dd
\mu(z)<\infty.\nonumber\\
&\text{Assumption \eqref{muassumption1} holds with $M>0$.}
\label{muassumption2}\tag{$\textup{A}_{\mu}^+$}
\end{align}

\begin{remark}
Without loss of generality, we can assume $f(0)=0$ and
  $\varphi(0)=0$ (by adding constants to $f$ and $\varphi$) and $f$
  and $\varphi$ 
  are globally Lipschitz (since solutions are
  bounded). \eqref{muassumption1} implies that
  $\int_{|z|>0}|z|^2\wedge 1\,\dd\mu(z)<\infty$ and $\mu$ is a L\'evy measure. \!\!\! \!\!\!
\label{assumptionremark}
\end{remark}

Equation \eqref{E} is a degenerate parabolic equation. It can be
strongly degenerate, i.e. $\varphi'$ may vanish/degenerate on sets of
positive measure. Equation \eqref{E} can therefore be of mixed
hyperbolic parabolic type. 
The equation is local when
$\mathfrak{L}=\Delta$ and non-local when $\mathfrak{L}=\mathcal
L^\mu$. In the latter case, it is an anomalous diffusion equation:
When \eqref{muassumption1} holds, $\mathcal{L}^{\mu}$ is the generator
of a pure jump L\'evy process, and conversely, any pure jump L\'evy
process has a generator like $\mathcal{L}^{\mu}$. An example is the isotropic
$\alpha$-stable process for $\alpha\in(0,2)$. Here the generator is the
fractional Laplacian $-(-\Delta)^{\frac{\alpha}{2}}$, which can be
defined as a Fourier multiplier, or equivalently, via
\eqref{defgenerator} with $\dd\mu(z)=c_\alp\frac{\dd
  z}{|z|^{d+\alpha}}$ for some $c_\alp>0$ \cite{App09, DrIm06}. If
also \eqref{muassumption2} holds, then 
$\mathcal{L}^{\mu}$ is the generator of a tempered $\alpha$-stable
process \cite{CoTa04}. Almost all L\'evy processes in finance are of
this type. In this paper, this assumption is needed to ensure that the solution of a dual problem belongs to $L^1$; see the discussion on page 3. For more details and examples of non-local operators,
we refer to \cite{App09, CoTa04}.

A large number of physical and financial problems are modeled by
convection-diffusion equations like \eqref{E}. Being very selective we
mention  reservoir simulation
\cite{EsKa00}, sedimentation processes \cite{BuCoBuTo99}, and traffic
flow \cite{Whi74} in the local case; detonation in gases
\cite{Cla02}, radiation hydrodynamics \cite{RoYo07, Ros89}, and
semiconductor growth \cite{Woy01} in the non-local case; and porous
media flow \cite{Vaz07,DeQuRoVa12} and mathematical finance
\cite{CoTa04} in both cases.  

Let us give the main references for the well-posedness of the Cauchy
problem for \eqref{E}, starting with the most classical case
$\mathfrak{L}=\Delta$. For a more complete bibliography, see the  
books \cite{Dib93,Daf10,Vaz07} and the references in \cite{KaRi03}. In
the hyperbolic case where $\varphi' \equiv 0$, we get the scalar
conservation law $\partial_t u+\Div f(u)=0$. The solutions of this
equation can develop discontinuities in finite time and the weak
solutions of the Cauchy problem are generally not unique. 
The most famous uniqueness result relies on the notion of entropy
solutions introduced in \cite{Kru70}.  In the pure diffusive case
where $f' \equiv 0$, there is no more creation of shocks and the
initial-value problem for $\partial_t u-\triangle \varphi(u)=0$ admits
a unique weak solution, cf. \cite{BrCr79}. Much later, the adequate
notion of entropy solutions for mixed hyperbolic parabolic equations
was introduced in~\cite{Car99}.
This paper focuses on an initial-boundary value problem.
For a general well-posedness result applying to the Cauchy problem
\eqref{E} with $\mathfrak{L}=\Delta$, we refer to e.g. \cite{KaRi03}
and \cite{AnMa10,MaTo03}.  

At the same time, there has been a large interest in non-local versions
of these equations (where $\mathfrak{L}=\mathcal L^\mu$). The study of
non-local diffusion terms was probably initiated by
\cite{BiFuWo98}. Now, the well-posedness is quite 
well-understood in the non-degenerate linear case where
$\varphi(u)=u$. Smooth solutions exist and are unique for subcritical
equations~\cite{BiFuWo98,DrGaVo03}, shocks can occur~\cite{AlDrVo07,KiNaSh08}
and weak solutions can be non-unique~\cite{AlAn10} for supercritical
equations, entropy solutions exist and are always
unique~\cite{Ali07,KaUl11}; cf. also e.g. \cite{ChCzSi10} for original
regularizing effects.  Very recently, the well-posedness theory of
entropy solutions was extended in~\cite{CiJa11} to cover the full
problem~\eqref{E}, even for strongly degenerate~$\varphi$. See also
\cite{DeQuRoVa12,BoVa13} on fractional porous medium type equations.

In all the papers on entropy solutions, the authors use 
doubling of variables arguments inspired by Kru\v{z}kov to prove 
$L^1$ contraction estimates. For entropy solutions $u$ and $v$,
the typical estimate when $g=0$ is
\begin{align}
\label{L1contr0}\int_{\R^d} (u(x,t)-v(x,t))^+ \dd x \leq \int_{\R^d}(u(x,0)-v(x,0))^+
\dd x.
\end{align}
From such an estimate the maximum or comparison principle follows: If
$u(x,0)\leq v(x,0)$ a.e., then $u(x,t)\leq v(x,t)$ for all $t>0$ and a.e.
$x$. A priori estimates for the $L^1$, $L^\infty$, and $BV$ norms of
the solutions also follow, estimates which are important e.g. to show
existence, stability, and convergence of approximations. However,
due to the global nature of this contraction estimate, 
it only applies for entropy solutions which satisfy $(u(\cdot,0)-v(\cdot,0))^+\in L^1(\mathbb{R}^d)$. In particular, this estimate cannot be used to obtain 
$L^1$ or $BV$ type estimates when $u(\cdot,0)$ and $v(\cdot,0)$ merely
belong to $L^\infty$ as in this paper. Some of the previous results also need
the further restriction that solutions belong to
$L^1\cap L^\infty$, see \cite{KaRi03,CiJa11}. 
In particular, prior to this paper, there were no well-posedness
results for merely bounded solutions of the non-local variant of
\eqref{E} when  $\varphi$ is non-linear.  

In this paper, we obtain new $L^1$ contraction results for \eqref{E}.
The estimates are more local than 
\eqref{L1contr0} and take the form of a ``partial Duhamel formula''
(see equation \eqref{L1contr3}), 
\begin{equation}
\label{L1contr}
\begin{split}
&\int_{B(x_0, \, M)}(u(x,t)-v(x,t))^+\dd x\leq\int_{B(x_0,M+1+Lt)}\big[\tilde\Phi(\cdot,t)\ast\big(u(\cdot,0)-v(\cdot,0)\big)^+\big](x)\dd x,
\end{split}
\end{equation}
for all $x_0\in\R^d$ and $M>0$, some $L$, and some integrable
function $\tilde\Phi$. See Section \ref{sec:main} for the precise
statements. In \eqref{L1contr}, there is no need to take $(u(\cdot,0)-v(\cdot,0))^+\in L^1(\mathbb{R}^d)$, and we will prove that the result applies to arbitrary bounded entropy solutions
$u,v$. In addition to this new and more quantiative form of the
$L^1$ contraction, we obtain as consequences new maximum/comparison principles
and BV estimates for both local and non-local versions of \eqref{E}, and in the non-local case, we obtain the first 
well-posedness result to hold for merely bounded entropy solution of \eqref{E}.

Estimate \eqref{L1contr} can be seen as a quantitative extension of the
finite speed of propagation type of estimate that holds for scalar
conservation laws \cite{Kru70,Daf10}. A similar (Duhamel type) result
has already been obtained for fractional conservation laws in
\cite{Ali07}. 
See also \cite{DrGaVo03,DrIm06} for more Duhamel formulas for fractional
conservation laws. The proof in \cite{Ali07} consists in establishing a
so-called Kato inequality for the equation, making a clever
choice of the test function to have cancellations, and then conclude in a
fairly standard way. Even if it is not written like that, the test
function is chosen to be a subsolution of a 
sort of dual equation that appears from the Kato inequality. In
\cite{Ali07}, the principal part of the ``dual equation'' is the
 (linear) fractional heat equation which can be solved exactly 
using the fundamental solution. The test function is therefore defined
via a Duhamel like formula involving the fractional heat kernel (the
function $\tilde\Phi$ in this case). 

In this paper, we formalize this proceedure and apply it to the more
difficult problems with non-linear degenerate diffusions. To do that,
we derive Kato inequalities for bounded entropy solutions and identify
the useful ``dual equations'' from them. In the general case, we find 
that the ``dual equations'' are fully non-linear degenerate parabolic
equations. These equations do not have smooth solutions in general,
but we then prove that there exist bounded continuous generalized
solutions (viscosity solutions) that belong to $L^1$. In this step, assumption \eqref{muassumption2} is needed in the non-local case. After several
regularization proceedures and Duhamel type of formulas, we produce a
test function that gives the necessary cancelations. Since this test
function is not based on a fundamental solution, or any $\tilde\Phi$
which is mass preserving, we can only conclude after additional
approximation steps.  

In effect, we have introduced a new way of obtaining $L^1$ contraction
estimates for degenerate parabolic equations. The new proof exploits a
``dual equation'' which in this case is pretty bad too, a degenerate
fully non-linear equation that can be best analyzed through the theory
of viscosity solutions \cite{CrIsLi92}. The proof can therefore be seen as
a sort of duality argument, and it is as far we know, the first proof
were viscosity solution methods were used as a key ingredient in a
contraction proof for entropy solutions.

The rest of this paper is organised as follows: In Section 2, we give
the definitions of entropy solutions and present and discuss our main
results. Their main consequences are discussed in Section 3. In
Section 4, we derive Kato type and other auxiliary inequalities. And
finally, in Section 5, we give the proofs of our main results. 

\subsection*{Notation} 
For $x\in\mathbb{R}$, we let $x^+=\, \text{max}\{x,0\}$, $x^-=(-x)^+$,
and $\sgn(x)$ is $\pm1$ for $\pm x>0$ and $0$ for $x=0$. We let
$B(x,r)=\{y\in\mathbb{R}^d : |x-y|< r\}$, and the indicator function
$\mathbf{1}_{\mathcal{A}}$ is $1$ on the set $\mathcal{A}$ and $0$ 
on the complement $\mathcal{A}^C$. By $L_\phi$ and  
$\text{supp}\,\phi$ we denote the Lipschitz constant and
support of a function $\phi$, derivatives are denoted by $'$,
$\frac{d}{dt}$, $\dell_{x_i}$, and $D\phi$ and $D^2\phi$ denote the $x$-gradient and Hessian matrix of $\phi$. Convolution is defined as
$f\ast g(x)=\left[f\ast g\right](x)=\int_{\R^d}f(x-y)g(y)\dd y$ (the brackets are dropped whenever the notation is not ambiguous). If $\mu$ is a Borel measure, then
$\mu^*$ is  defined as $\mu^*(B)=\mu(-B)$ for all Borel sets on
$\mathbb{R}^d\setminus \{0\}$. The $L^2$ adjoint of an operator $A$ is denoted by $A^*$, and the reader may check that
$(\mathcal L^{\mu})^*=\mathcal L^{\mu^*}$.

We use standard notation for $L^p$, $BV$, and $H^1$ spaces, $C_b$ and
$C_c^\infty$ are the spaces of bounded continuous functions and smooth
functions with compact support.  We use the following norm and semi-norm:
\begin{align*}
&  \|\phi\|_{C([0,T];L^1(\R^d))}:=\esssup_{t\in
  [0,T]}\int_{\R^d}|\phi(x,t)|\dd x,\\
& |\psi|_{BV(\R^d)}:= 
\sup_{h\neq0}\int_{\R^d}\frac{|\psi(x+h)-\psi(x)|}{|h|}\dd x. 
\end{align*}
The $|\cdot|_{BV}$ semi-norm is equivalent to standard definition of
the total variation, see \cite[Lemma A.1]{HoRi07} or \cite[Lemma
A.2]{AlCiJa12}. We define the spaces $C([0,T]; L^1(\mathbb{R}^d))$ and $C([0,T]; L_\text{loc}^1(\mathbb{R}^d))$ in the usual way. E.g., the space $C([0,T]; L_\text{loc}^1(\mathbb{R}^d))$ is  the space of measurable functions $u:\R^d\times[0,T]\to\R$ satisfying $u(\cdot,t)\in
L^1_\text{loc}(\mathbb{R}^d)$ for every $t\in[0,T]$,
$\max_{t\in[0,T]}\int_K|u(x,t)|\dd x <\infty$,  and
$\int_K|u(x,t)-u(x,s)|\dd x \to0$ when $t\to s$ for all compact
$K\subset \mathbb{R}^d$ and $s\in[0,T]$. 

For the rest of the paper, we fix three families of mollifiers
$\omega_\veps$, $\hat\omega_\veps$, $\rho_\veps$ defined by
\begin{align}
\omega_\varepsilon(\sigma):=\frac{1}{\varepsilon}\omega\left(\frac{\sigma}{\varepsilon}\right)
\label{mollifierspace}
\end{align}
for fixed $0\leq\omega\in C^\infty_c(\R)$ satisfying $\text{supp}\,
\omega\subseteq [-1,1]$, $\omega(\sigma)=\omega(-\sigma),$ 
$\int\omega=1$;
\begin{equation}
\hat{\omega}(x)=\omega(x_1)\dots\omega(x_d)\qquad\text{and}\qquad \hat{\omega}_\veps(x)=\frac{1}{\varepsilon^{d}}\hat{\omega}\Big(\frac{x}{\varepsilon}\Big)
\label{domega}
\end{equation}
for $x=(x_1,\ldots,x_d)\in\R^d$; and
\begin{align}
&\rho_\delta(\sigma,\tau):=\frac{1}{\delta^{d+2}}\rho\left(\frac{\sigma}{\delta},\frac{\tau}{\delta^2}\right)
\label{mollifierspacetime}
\end{align} 
for fixed $0\leq\rho\in C_c^\infty(Q_T)$, $\textup{supp}\,\rho\subseteq
B(0,1)\times(0,1)$, $\rho(\sigma,\tau)=\rho(-\sigma,-\tau)$,
$\int\rho=1$.

\section{Entropy formulation and main results}
\label{sec:main}
In this section, we give the definitions of entropy solutions of \eqref{E}
and then present our main results. We will use the following splitting
\begin{equation*}
\mathcal{L}^\mu[\phi](x)=\mathcal{L}_r^\mu[\phi](x)+\mathcal{L}^{\mu,r}[\phi](x)+b^{\mu,r}\cdot D\phi(x),
\end{equation*}
for $\phi\in C_c^{\infty}(Q_T)$, $r>0$ and $x\in \mathbb{R}^d$, where
\begin{equation*}
\begin{split}
\mathcal{L}_r^\mu[\phi](x):=&\int_{0<|z|\leq r} \phi(x+z)-\phi(x)-z\cdot D\phi \indikator \dd \mu(z),\\
\mathcal{L}^{\mu,r}[\phi](x):=&\int_{|z|>r}\phi(x+z)-\phi(x)\dd \mu(z),\\
b^{\mu,r}:=&-\int_{|z|>r} z\indikator\dd \mu(z).
\end{split}
\end{equation*}
Below we will use the Kru\v{z}kov entropy-entropy flux pairs, $|u-k|$ and
$\text{sign}(u-k)(f(u)-f(k))$, and the
corresponding semi entropy-entropy flux pairs,
\begin{equation*}
(u-k)^\pm\qquad\text{and}\qquad \pm\textup{sign}(u-k)^\pm(f(u)-f(k))
\qquad  \text{for all } k\in\mathbb{R}.
\end{equation*}

\begin{definition}[Entropy solutions]
Let $\mathfrak{L}=\Delta$. A function $u\in L^{\infty}(Q_T)\cap
C([0,T];L_\textup{loc}^{1}(\R^d))$ is  
\begin{enumerate}[(a)]
\item an entropy subsolution of \eqref{E} if
\begin{enumerate}[i)]
\item for all non-negative $\phi\in C_c^\infty(Q_T)$ and all $k\in\R$
\begin{equation}
\begin{split}
&\iint_{Q_T}(u-k)^+\phi_t+\sgn(u-k)^+[f(u)-f(k)]\cdot D\phi\,\dd x \dd t\\
&+\iint_{Q_T}(\varphi(u)-\varphi(k))^+\Delta\phi\,\dd x \dd t\\
&+\iint_{Q_T}\sgn(u-k)^+g\,\phi\, \dd x \dd t\geq0;
\end{split}
\label{defentsubsolnlap}
\end{equation}
\item $\varphi(u)\in L^2((0,T);H_\textup{loc}^1(\R^d))$;
\item  $u(\cdot,0)\leq u_0$ for a.e. $x\in\R^d$.
\end{enumerate}
\item an entropy supersolution of \eqref{E} if
\begin{enumerate}[i)]
\item for all non-negative $\phi\in C_c^\infty(Q_T)$ and all $k\in\R$
\begin{equation}
\begin{split}
&\iint_{Q_T}(u-k)^-\phi_t-\sgn(u-k)^-[f(u)-f(k)]\cdot D\phi\,\dd x \dd t\\
&+\iint_{Q_T}(\varphi(u)-\varphi(k))^-\Delta\phi\,\dd x \dd t\\
&+\iint_{Q_T}-\sgn(u-k)^-g\,\phi\, \dd x \dd t\geq0;
\end{split}
\label{defentsupersolnlap}
\end{equation}
\item $\varphi(u)\in L^2((0,T);H_\textup{loc}^1(\R^d))$;
\item $u(\cdot,0)\geq u_0$ for a.e. $x\in\R^d$.
\end{enumerate}
\item an entropy solution of \eqref{E} if it is both and entropy subsolution and an entropy supersolution.
\end{enumerate}
\label{defentsolnlap}
\end{definition}

\begin{definition}[Entropy solutions]
Let $\mathfrak{L}=\mathcal{L}^\mu$. A function $u\in L^\infty(Q_T)\cap C([0,T];L_\textup{loc}^1(\R^d))$ is
\begin{enumerate}[(a)]
\item an entropy subsolution of \eqref{E} if
\begin{enumerate}[i)]
\item for all non-negative $\phi\in C_c^{\infty}(Q_T)$ and all $k\in\mathbb{R}$
\begin{equation}
\begin{split}
&\iint_{Q_T}(u-k)^+\dell_t\phi+\sgn(u-k)^+[f(u)-f(k)]\cdot D\phi\,\dd x \dd t\\
&+\iint_{Q_T}(\varphi(u)-\varphi(k))^+\left(\mathcal{L}_r^{\mu^*}[\phi]+b^{\mu^*,r}\cdot D\phi\right)+\sgn(u-k)^+\mathcal{L}^{\mu,r}[\varphi(u)]\phi \,\dd x \dd t\\
&+\iint_{Q_T}\sgn(u-k)^+g\,\phi\, \dd x \dd t\geq0;
\end{split}
\label{defentsubsoln}
\end{equation}
\item $u(\cdot,0)\leq u_0(\cdot)$ for a.e. $x\in\mathbb{R}^d$.
\end{enumerate}
\item an entropy supersolution of \eqref{E} if
\begin{enumerate}[i)]
\item for all non-negative $\phi\in C_c^{\infty}(Q_T)$ and all $k\in\mathbb{R}$
\begin{equation}
\begin{split}
&\iint_{Q_T}(u-k)^-\dell_t\phi-\sgn(u-k)^-[f(u)-f(k)]\cdot D\phi\,\dd x \dd t\\
&+\iint_{Q_T}(\varphi(u)-\varphi(k))^-\left(\mathcal{L}_r^{\mu^*}[\phi]+b^{\mu^*,r}\cdot D\phi\right)-\sgn(u-k)^-\mathcal{L}^{\mu,r}[\varphi(u)]\phi \,\dd x \dd t\\
&+\iint_{Q_T}-\sgn(u-k)^-g\,\phi\, \dd x \dd t\geq0;
\end{split}
\label{defentsupersoln}
\end{equation}
\item $u(\cdot,0)\geq u_0(\cdot)$ for a.e. $x\in\mathbb{R}^d$.
\end{enumerate}
\item an entropy solution of \eqref{E} if it is both an entropy subsolution and an entropy supersolution.
\end{enumerate}
\label{defentsoln}
\end{definition}

\begin{remark}
\begin{enumerate}[(a)]
\item Similar definitions are given e.g. in \cite[Definition
  3.4]{MaTo03} and \cite[Definition 5.1]{CiJa11}.
\item Since an entropy solution $u\in C([0,T];L_\text{loc}^1(\mathbb{R}^d))$ and $u(\cdot,0)=u_0(\cdot)$ a.e., the initial condition is imposed in a strong sense: $u(\cdot,t)\to u_0(\cdot)$ in $L_\textup{loc}^1$ as $t\to0^+$.
\item By \eqref{fassumption}, \eqref{Aassumption}, and $u\in
  L^{\infty}(Q_T)$,  $f(u)$ and $\varphi(u)$ are in $L^{\infty}(Q_T)$. 
\item By c) and \eqref{gassumption}, all
  integrals in \eqref{defentsubsolnlap} and
  \eqref{defentsupersolnlap} are well-defined. 
\item By c) and \eqref{gassumption}, the first and
  third integral in 
  \eqref{defentsubsoln} and \eqref{defentsupersoln} are
  well-defined. Since $\mathcal{L}_r^{\mu^*}[\phi]\in
  C_c^{\infty}(Q_T)$ for $\phi\in C_c^{\infty}(Q_T)$ and
  $\mathcal{L}^{\mu,r}[\varphi(u)]\in L^{\infty}(Q_T)$ for
  $\varphi(u)\in L^{\infty}(Q_T)$, then by c) the second integral is
  also well-defined. Since $u$ is a Lebesgue measurable function,  it
  is not immediatly clear that $\varphi(u)$ is $\mu$-measurable and
  $\mathcal{L}^{\mu,r}[\varphi(u)]$ is point-wisely well-defined. We refer to
  Remark 2.1 and Lemma 4.2 in \cite{AlCiJa12} for a discussion and
  proof that this is actually the case. 
\end{enumerate}
\label{defentsolnremark}
\end{remark}

\begin{lemma}
$u(x,t)$ is an entropy solution of \eqref{E} in the sense of
Definition \ref{defentsolnlap} or \ref{defentsoln}  if and only if
$u(x,t)$ is an entropy solution in the usual sense. 
\label{solutioniffsubsuper}
\end{lemma}

\begin{proof}
Since $|u-k|=(u-k)^++(u-k)^-$ and $\sgn(u-k)=\sgn(u-k)^+-\sgn(u-k)^-$,
\begin{gather*}
\eqref{defentsubsolnlap}+\eqref{defentsupersolnlap}\qquad\text{or}\qquad\eqref{defentsubsoln}+\eqref{defentsupersoln}
\\
\Downarrow\\
|u-k|_t+\diver \Big(\sgn(u-k)[f(u)-f(k)]\Big)-\mathfrak{L}\,\big|\varphi(u)-\varphi(k)\big|-\sgn(u-k)g\leq 0
\end{gather*} 
in $\mathcal{D'}(Q_T)$, which is the usual definition in terms of Kru\v{z}kov entropy-entropy fluxes.

Part a) of Definitions \ref{defentsoln} and \ref{defentsolnlap} can be
 obtained from the usual definition in a similar way. First we check that
 $u-k$ satify
$$(u-k)_t+\diver\,\big(
f(u)-f(k)\big)-\mathfrak{L}\big(\varphi(u)-\varphi(k)\big)-g=0\quad
\text{in}\quad \mathcal{D}'(Q_T). $$
Then we add this equation to the entropy inequality for $u$. Since
this inequality involves the Kru\v{z}kov flux $|u-k|$, the result follows by the
 following identities 
\begin{align*}
&|u-k|+(u-k)=2(u-k)^+,\\
&\sgn(u-k)(f(u)-f(k))+\big(f(u)-f(k)\big)=2\sgn(u-k)^+(f(u)-f(k)),
\end{align*}
and a similar one for the $\varphi(u)$-terms. The proof of part b) is similar.
\end{proof}

\subsection*{Main results}
To give the main results, we introduce the functions $\tilde{K}$ and
$\Phi$. We define 
\begin{equation}
\tilde{K}(x,t)=\mathcal{F}^{-1}(\e^{-t|2\pi\xi|^\alpha})(x) \quad \text{for $\alpha\in(0,2]$}, 
\label{heatkernel}
\end{equation}
where $\mathcal{F}(\phi)(\xi)=\int_{\R^d}\e^{-2\pi\textup{i} \xi\cdot
  x}\phi(x)\dd x$. Then $\tilde{K}$ is a fundamental solution satifying
\begin{equation*}
\begin{cases}\dell_t \tilde{K}-\mathfrak{L^*}\tilde{K}=0, & t>0,\\[0.2cm]
\tilde{K}(x,0)=\delta_0,&
\end{cases}
\end{equation*}
for $\mathfrak{L}^*=\mathfrak{L}=-(-\Delta)^{\frac{\alpha}{2}}$, where
$\delta_0$ is the Dirac measure centred at the origin. Furthermore,
$\Phi$ is the (non-smooth viscosity) solution of  
\begin{equation}
\begin{cases}\dell_t \Phi-(\mathfrak{L}^{*}\Phi)^+=0 & \text{in}\quad\R^d\times(0,\tilde{T}),\\[0.2cm]
\Phi(x,0)=\Phi_0(x) & \text{on}\quad\R^d,
\end{cases}
\label{viscositysoln}
\end{equation} 
for some $\Phi_0\in C_c^{\infty}(\R^d)$. 

\begin{lemma}
Let $\tilde{K}$ be defined by \eqref{heatkernel}, then it has the following properties
\begin{enumerate}[i)]
\item $\tilde{K}$ is non-negative, smooth, and bounded for $t>\delta$ for all $\delta>0$;
\item $\int_{\mathbb{R}^d}\tilde{K}(x,t)\dd  x =1$;
\item $\{\tilde{K}(\cdot,t)\}_{t>0}$ is an approximate unit as $t\to0$;
\item $\tilde{K}(x,t)=\tilde{K}(-x,t)$ for all $t>0$ and $x\in\mathbb{R}^d$.
\end{enumerate}
\label{propofheatkernel}
\end{lemma}

This result is classical and can be found in e.g. \cite{Ali07}.

\begin{lemma}\label{Phiprop}
  Assume \eqref{fassumption}, \eqref{Aassumption}, \eqref{gassumption}
  hold, that $\mathfrak{L}=\Delta$ or $\mathfrak{L}=\mathcal{L}^\mu$ and
  \eqref{muassumption2} holds, and that $0\leq \Phi_0\in
  C_c^\infty(Q_T)$. Let $\tilde{T}:=\max\{T,L_\varphi
  T\}$ where $L_\varphi$ is the Lipschitz constant of $\varphi$. Then
  there exists a unique viscosity solution $\Phi(x,t)$ of
  \eqref{viscositysoln} such that 
\begin{equation*}
0\leq \Phi\in C_b(Q_{\tilde{T}})\cap C([0,\tilde{T}];L^1(\R^d)).
\end{equation*}
\end{lemma}

We prove this lemma in Section \ref{sectionlocalcon}. Note that
viscosity solutions are the right type of weak solutions for fully
non-linear and degenerate equations like \eqref{viscositysoln}, see e.g. \cite{CrIsLi92,JaKa05}.

\begin{remark}
\begin{enumerate}[(a)]
\item To handle bounded, non-integrable solutions of \eqref{E}, it is
  important that $\Phi$ belongs to $L^1$  -- a non-standard
  result for equation \eqref{viscositysoln}.
\item As for $\tilde K$, we would have liked to take
  $\Phi_0=\delta_0$ (Dirac measure), since this would give us better constants in the results that
follow. We have not been able to do it for two reasons: i) There is
no well-posedness theory for equations like \eqref{viscositysoln} with
measure initial data, and ii) the $L^1$ bound for $\Phi$ is
obtained by comparison with a particular $L^1$ supersolution. Hence, if
we let $\Phi_0$ be an approximate delta function and then took the
limit, these estimates would blow up and the cruicial $L^1$ property
would be lost. 
\item When $\mathfrak{L}$
is self-adjoint (that is, when $\mathfrak{L}=\Delta$ or
$\mathfrak{L}=\mathcal{L}^{\mu}$ with $\mu$ 
symmetric), we may assume that $\Phi(-x,t)=\Phi(x,t)$. Simply take a
symmetric $\Phi_0$ and the solution of
\eqref{viscositysoln} has this property. 
\label{Phi_deltasymm}
\end{enumerate}
\end{remark}

Before the main theorems are given, we revisit some of the known
results in special cases.   

\begin{theorem}
Assume \eqref{fassumption} holds, and $\varphi=0$. Let $u$ and $v$ be
entropy sub- and supersolutions of \eqref{E} with initial data $u_0,
v_0\in L^\infty(\R^d)$ and measurable source terms $g,h$ satisfying $\int_0^T\|g(\cdot,t)\|_{L^\infty(\mathbb{R}^d)}+\|h(\cdot,t)\|_{L^\infty(\mathbb{R}^d)}\dd t<\infty$. Then for all $t\in(0,T)$,
$M>0$ and $x_0\in\mathbb{R}^d$ 
\begin{equation*}
\begin{split}
\int_{B(x_0, \, M)}(u(x,t)-v(x,t))^+\dd x\leq \int_{B(x_0, \,M+L_ft)}(u_0(x)-v_0(x))^+\dd x\\
\qquad+\int_0^t\int_{B(x_0,\,M+L_f(t-s))}(g(x,s)-h(x,s))^+\dd
x \dd s, 
\end{split}
\end{equation*}
where $L_f$ is the Lipschitz constant of $f$.
\label{localDafcontraction}
\end{theorem}

This is the classical local $L^1$ contraction result for scalar
conservation laws, see e.g. Dafermos \cite[p. 149]{Daf10} for a
proof. The hyperbolic finite speed of propagation property is encoded
in the result.

In the linear non-local diffusion case, Alibaud \cite{Ali07} obtained
the inequality
\begin{align}
\label{L1contr2}
\int_{B(x_0, \, M)}(u(x,t)-v(x,t))^+\dd x\leq \int_{B(x_0, \,M+L_ft)}\big[\tilde{K}(\cdot,t)\ast(u_0-v_0)^+\big](x)\dd x\\
\quad+\int_0^t\int_{B(x_0,\,M+L_f(t-s))}\big[\tilde{K}(\cdot,t-s)\ast(g(\cdot,s)-h(\cdot,s))^+\big](x)\dd x \dd s,\nonumber
\end{align}
where $L_f$ is the Lipschitz constant of $f$. We state the result
along with a new result for the local case.

\begin{theorem}\label{localKcontraction}
Assume \eqref{fassumption}, $\varphi(u)=u$, and  $\tilde{K}$ is defined by
\eqref{heatkernel}. Let $t\in(0,T)$, $M>0$, 
$x_0\in\mathbb{R}^d$, and $u$
and $v$ be entropy sub- and supersolutions of \eqref{E} with initial data
$u_0, v_0\in L^\infty(\R^d)$ and measurable source terms $g,h$
satisfying $\int_0^T\|g(\cdot,t)\|_{L^\infty(\mathbb{R}^d)}+\|h(\cdot,t)\|_{L^\infty(\mathbb{R}^d)}\dd t<\infty$. 
\begin{enumerate}[(a)]
\item  If $\mathfrak{L}=-(-\Delta)^{\frac{\alpha}{2}}$ for $\alpha\in(0,2)$, then the $L^1$ contraction estimate
\eqref{L1contr2} holds. 
\medskip
\item If $\mathfrak{L}=\Delta$ ($\alp=2$), then the $L^1$ contraction estimate
\eqref{L1contr2} holds. 
\end{enumerate}
\end{theorem}

The result has the form of a
partial Duhamel formula involving the fundamental solution of
the parabolic part of the equation (which is linear here).
The proof of (a) can be found in \cite{Ali07} when $g=0$, and
the extension to general $g$ is easy. 
Part (b) seems to be new, but essentially it follows from the
argument of \cite{Ali07} and Proposition \ref{dualequation}. The proof
is given in Section \ref{sectionlocalcon}.

Now, we give our main result which is an $L^1$ contraction estimate of
the form 
\begin{align}\label{L1contr3}
\int_{B(x_0, \, M)}(u(x,t)-v(x,t))^+\dd x\leq\int_{B(x_0,M+1+L_ft)}\big[\Phi(-\cdot,L_\varphi t)\ast(u_0-v_0)^+\big](x)\dd x\\
\quad+\int_0^t\int_{B(x_0,M+1+L_f(t-s))}\big[\Phi(-\cdot,L_\varphi(t-s))\ast(g(\cdot,s)-h(\cdot,s))^+\big](x)\dd x \dd s,\nonumber
\end{align}
where $L_f$ and $L_\varphi$ are the Lipschitz constants of $f$ and
$\varphi$ respectively.

\begin{theorem}\label{localcontractions}
Assume \eqref{fassumption}, \eqref{Aassumption} hold, and $\Phi$ is
given by Lemma \ref{Phiprop}. Let $t\in(0,T)$,
$M>0$, $x_0\in\mathbb{R}^d$, and $u$ and $v$ be entropy sub- and
supersolutions of 
\eqref{E}  with initial data
$u_0, v_0\in L^\infty(\R^d)$ and measurable source terms $g,h$
satisfying $\int_0^T\|g(\cdot,t)\|_{L^\infty(\mathbb{R}^d)}+\|h(\cdot,t)\|_{L^\infty(\mathbb{R}^d)}\dd t<\infty$.  
\begin{enumerate}[(a)]
\item  If $\mathfrak{L}=\mathcal L^\mu$ and \eqref{muassumption2}
  holds, then the $L^1$ contraction estimate 
\eqref{L1contr3} holds. 
\medskip
\item If $\mathfrak{L}=\Delta$, then the $L^1$ contraction estimate
\eqref{L1contr3} holds. 
\end{enumerate}
\end{theorem}
The proof is given in Section \ref{sectionlocalcon}. These results,
the $L^1$ contractions \eqref{L1contr2} and \eqref{L1contr3}, encode
both the finite speed of propagation of the hyperbolic term and the
infinite speed of propagation of the parabolic term. 
As far as we know, this is the first time such a partial Duhamel type
$L^1$ contraction result has been given for non-linear diffusions.

\begin{remark}\begin{enumerate}[(a)]
\item By Fubini and a change of variables\footnote{E.g.\begin{align*}
&\int_{B(x_0,M+1+L_ft)}\int_{\R^d}\Phi(-y,L_\varphi
t)(u_0-v_0)^+(x-y)\dd y\dd x\\
& =\int_{\R^d}\Phi(-y,L_\varphi
t)\int_{B(x_0,M+1+L_ft)}(u_0-v_0)^+(x-y)\dd x\dd y \\
&=\int_{\R^d}\Phi(-y,L_\varphi
t)\int_{B(x_0-y,M+1+L_ft)}(u_0-v_0)^+(z)\dd z\dd y 
\end{align*}}, the $L^1$ contraction
  \eqref{L1contr3} is equivalent to an inequality involving
  convolutions of local $L^1$ norms and $\Phi$: 
\begin{align}
\nonumber
\|(u(\cdot,t)-v(\cdot,t))^+\|_{L^1(B(x_0,M))}\leq
\int_{\R^d}\Phi(-y,L_\varphi
t)\|(u_0-v_0)^+\|_{L^1(B(x_0-y,M+1+L_ft))}\dd y\!\!\!\!\!\!\!\\
\label{L1contr3b}\quad+\int_0^t\int_{\R^d}\Phi(-y,L_\varphi(t-s))\|(g(\cdot,s)-h(\cdot,s))^+\|_{L^1(B(x_0-y,M+1+L_f(t-s)))}\dd
y\dd s.\!\!\!\!\!\!\!
\end{align}
\item Theorem \ref{localcontractions} gives a stronger
  $L^1$ contraction estimate than previous results \cite{MaTo03,AnMa10,CiJa11},
  see the discussion in the introduction and the next section.    
\item Theorem \ref{localcontractions}
  (a) is the first $L^1$ contraction result for bounded solutions of
  \eqref{E} with non-local $\mathfrak L$. 
\item Theorem \ref{localcontractions} (a) holds under assumption
  \eqref{muassumption2} which is discussed in 
  the introduction. We do not know if this assumption can be
  relaxed. We use it to prove that 
  $\Phi(\cdot,t)$ belongs to $L^1$, a result which is needed for
  \eqref{L1contr3} to be well-defined for merely bounded initial data and source
  term.  
\item The $+1$-factor in $B(x_0, M+1+L_ft)$ in Theorem
  \ref{localcontractions} depends on the 
choice of $\Phi$, and comes from the fact that $\Phi(x,t)$
is not an approximate unit as $t\to 0^+$. In fact, it will have
increasing mass (or $L^1$ norm) in time.    
\end{enumerate}
\end{remark}

\section{Consequences}
\label{sectionconsequences}

Using Theorem \ref{localcontractions}, we now derive maximum and comparison principles, new a priori estimates,
and new existence and uniqueness results for \eqref{E}. The latter results are new only in the non-local case.

\begin{corollary}
Assume \eqref{fassumption} and \eqref{Aassumption} hold,
\eqref{muassumption2} holds when
$\mathfrak{L}=\mathcal{L}^{\mu}$, $u_0, v_0\in~L^\infty(\R^d)$, and
measurable $g,h$ satisfying 
$$\int_0^T\|g(\cdot,t)\|_{L^\infty(\mathbb{R}^d)}+\|h(\cdot,t)\|_{L^\infty(\mathbb{R}^d)}\dd
t<\infty.$$
Let $M>0$, $x_0\in\mathbb{R}^d$ and $L_f$ and $L_\varphi$ be the Lipschitz
constants of $f$ and $\varphi$ respectively. 
\begin{enumerate}[(a)]
\item \textup{($L^1$ contraction).} Let $u$ and $v$ be entropy
  solutions of \eqref{E} with initial data $u_0, v_0$ and source terms
  $g, h$ respectively. Then for all $t\in(0,T)$,
\begin{equation*}
\begin{split}
\|u(\cdot,t)-v(\cdot,t)\|_{L^1(B(x_0,M))}\leq \|\Phi(-\cdot,L_\varphi t)\ast|u_0-v_0|\|_{L^1(B(x_0,M+1+L_ft))}\\
\quad+\int_0^t\|\Phi(-\cdot,L_\varphi(t-s))\ast|g(\cdot,s)-h(\cdot,s)|\|_{L^1(B(x_0,M+1+L_f(t-s)))}\dd s.
\end{split}
\end{equation*}
\item \textup{($L^{1}$ bound).} Let $u$ be an entropy solution of
  \eqref{E}. Then for all $t\in(0,T)$, 
\begin{equation*}
\begin{split}
\|u(\cdot,t)\|_{L^{1}(B(x_0,M))}\leq\|\Phi(-\cdot,L_\varphi t)\ast |u_0|\|_{L^{1}(B(x_0,M+1+L_ft))}\\
\quad+\int_0^t\|\Phi(-\cdot,L_\varphi(t-s))\ast|g(\cdot,s)|\|_{L^1(B(x_0,M+1+L_f(t-s)))}\dd s.
\end{split}
\end{equation*}
\item \textup{(Comparison principle).} Let $u$ and $v$ be entropy sub-
  and supersolutions of \eqref{E} with initial data $u_0, v_0$ and
  source terms $g,h$ respectively. If $u_0\leq v_0$ a.e. on
  $\R^d$ and $g\leq h$ a.e. in $Q_T$, then 
\begin{equation*}
u(x,t)\leq v(x,t)\qquad\text{a.e. in $Q_T$}.
\end{equation*}

\item \textup{(Maximum principle).} Let $u$ be an entropy solution of \eqref{E}. Then
\begin{equation*}
\inf_{x\in\R^d}u_0(x)+\int_0^t\inf_{x\in\R^d}g(x,s)\dd s \leq
u(x,t)\leq\sup_{x\in\R^d}u_0(x)+\int_0^t\sup_{x\in\R^d}g(x,s)\dd s 
\end{equation*}
a.e. in $Q_T$.
\item \textup{($BV$ bound).} Let $u$ be an entropy solution of
  \eqref{E} and assume $u_0\in BV(\mathbb{R}^d)$, $g$ is measurable,
  and $\int_0^T|g(\cdot, t)|_{BV(\mathbb{R}^d)}\dd t<\infty$. Then for all \\$t\in(0,T)$, $x_0\in\R^d$, and $M>0$,
\begin{equation*}
\begin{split}
&|u(\cdot,t)|_{BV(B(x_0,M))}\\
&\leq\sup_{h\neq0}\frac{\|\Phi(-\cdot,L_\varphi t)\ast|u_0(\cdot+h)-u_0|\|_{L^1(B(x_0,M+1+L_ft))}}{|h|}\\
&\quad+\sup_{h\neq0}\frac{\int_0^t\|\Phi(-\cdot,L_\varphi(t-s))\ast|g(\cdot+h,s)-g(\cdot,s)|\|_{L^1(B(x_0,M+1+L_f(t-s)))}\dd s}{|h|}\\
\end{split}
\end{equation*}
\end{enumerate}
\label{collcorolocal}
\end{corollary}

\begin{remark}
The $L^1$ and $BV$ bounds are new even in the local
case. 

In a similar way as in \eqref{L1contr3b}, the bounds
in a), b), e) can be expressed as convolutions of local
norms. E.g. when $g=h=0$,
\begin{align*}
\|u(\cdot,t)\|_{L^1(B(x_0,M))} &\leq \int_{\R^d}\Phi(-y,L_\varphi
t) \|u_0\|_{L^1(B(x_0-y,M+1+L_ft))}\dd y\\
|u(\cdot,t)|_{BV(B(x_0,M))}
& \leq \int_{\R^d}\Phi(-y,L_\varphi
t) |u_0|_{BV(B(x_0-y,M+1+L_ft))}\dd y.
\end{align*}
If $|u_0|_{BV(\mathbb{R}^d)}<\infty$, it follows that $|u(\cdot,t)|_{BV(B(x_0,M))}\leq\|\Phi(\cdot,L_\varphi t)\|_{L^1(\mathbb{R}^d)}|u_0|_{BV(\mathbb{R}^d)}$.
\end{remark}
\begin{proof}
a) By  Theorem \ref{localcontractions}, estimate \eqref{L1contr3}
holds. Interchanging the roles of $u,g$ and $v,h$, and using 
$(v-u)^+=(u-v)^-$ etc., we see that \eqref{L1contr3} holds for $(u-v)^-$
as well as for $(u-v)^+$. Hence a) follows.  
\medskip

\noindent b) Follows from a) with $v=v_0=h=0$.
\medskip

\noindent c) By  the contraction estimate \eqref{L1contr3} and the assumptions on the
initial data and source terms, for all $t>0$, $x_0\in\R^d$, and $M>0$,
\begin{equation*}
\begin{split}
&\int_{B(x_0, \, M)}(u(x,t)-v(x,t))^+\dd x\leq 0.
\end{split}
\end{equation*}
Hence $(u-v)^+=0$ and $u\leq v$ a.e. in $Q_T$.
\medskip

\noindent d) Note that
$w(t)=\sup_{x\in\R^d}u_0(x)+\int_0^t\sup_{x\in\R^d}g(x,s)\dd s$ is an entropy
supersolution of \eqref{E}, and then $u\leq w$ a.e. by part c). In a
similar way, the lower bound follows.

\medskip
 
\noindent e) Since \eqref{E} is translation invariant, both $u(x,t)$ and
$u(x+h,t)$ are entropy solutions of \eqref{E} with initial data
$u_0(x)$ and $u_0(x+h)$, and sources $g(x,t)$ and $g(x+h,t)$ respectively. By the definition of $|\cdot|_{BV}$ and part a),
\begin{equation*}
\begin{split}
&|u(\cdot,t)|_{BV(B(x_0,M))}\\
&=\sup_{h\neq0}\frac{\|u(\cdot+h,t)-u(\cdot,t)\|_{L^1(B(x_0,M))}}{|h|}\\
&\leq\sup_{h\neq0}\int_{B(x_0,M+1+L_ft)}\int_{\mathbb{R}^d}\Phi(-(x-y),L_\varphi t)\frac{|u_0(y+h)-u_0(y)|}{|h|}\dd y \dd x\\
&\quad+\sup_{h\neq0}\int_0^t\int_{B(x_0,M+1+L_f(t-s))}\int_{\R^d}\Phi(-(x-y),L_\varphi(t-s))\\
&\qqquad\qqquad\qqquad\qqquad\qqquad\cdot\frac{|g(y+h,s)-g(y,s)|}{|h|}\dd y\dd x\dd s.
\end{split}
\end{equation*}
\end{proof}

\begin{theorem}[Existence and uniqueness]\label{existence}
Assume that \eqref{fassumption}, \eqref{gassumption},
\eqref{Aassumption}, and \eqref{u_0assumption} hold, and
\begin{equation*}
\mathfrak{L}=\Delta \quad\text{or}\quad\mathfrak{L}=\mathcal{L}^{\mu}\text{ and \eqref{muassumption2} holds}.
\end{equation*}
Then there exists a unique entropy solution of the initial value
  problem \eqref{E}.
\end{theorem}

\begin{proof} 
In the local case, this result was proved in \cite[Theorem
3.7]{MaTo03}. In the non-local case, uniqueness is an immediate
consequence of Theorem \ref{localcontractions} with $u_0=v_0$ and
$g=h$, and the existence result follows from existence
results for $L^1\cap L^\infty$ solutions \cite{CiJa11,CiJa14} and the
$L^1$ contraction of Corollary \ref{collcorolocal} a). We do the proof
under the simplifying assumption that $g=0$. It  
is  not hard to extend the proof to the general case. 

Take functions
$u_{0,n}\in L^\infty(\mathbb{R}^d)\cap L^1(\mathbb{R}^d)$ such that 
\begin{equation}
\|u_{0,n}\|_{L^\infty(\R^d)}\leq\|u_0\|_{L^\infty(\R^d)}\text{ and }u_{0,n}\to u_0\text{ in $L_\textup{loc}^1(\R^d)$ and pointwise a.e.} 
\label{approxofinitdata}
\end{equation}
By \cite{CiJa11,CiJa14}, there exist entropy
solutions $u_m, u_n$ of \eqref{E} with initial data
$u_{0,m},u_{0,n}$ respectively. By Corollary \ref{collcorolocal} a) and
the triangle inequality,  
\begin{equation*}
\begin{split}
&\|u_m-u_n\|_{C([0,T];L^1(B(x_0,M)))}\\
&\leq\max_{t\in[0,T]}\|\Phi(-\cdot,L_\varphi t)\ast|u_{0,m}-u_0|\|_{L^1(B(x_0,M+1+L_ft))}\\
&\quad+\max_{t\in[0,T]}\|\Phi(-\cdot,L_\varphi t)\ast|u_{0,n}-u_0|\|_{L^1(B(x_0,M+1+L_ft))}.
\end{split}
\end{equation*}
The right-hand side of the inequality goes to zero by Lebesgue's
dominated convergence theorem and \eqref{approxofinitdata} when
$n,m\to\infty$ (the integrand is dominated by $2\Phi(-y,L_\varphi
t)\|u_0\|_{L^\infty}$). Therefore, the sequence of entropy solutions
$\{u_n\}$ is Cauchy in $C([0,T];L^1(B(x_0,M)))$. 

Since $\R^d$ can be covered by a countable number of such balls, a
diagonal argument produces a function $u$ such that $u_\veps \to
u$ in $C([0,T];L_\textup{loc}^1(\R^d))$. Taking, if necessary, a
further subsequence we may assume $u_n\to u$ a.e., and hence
$\|u\|_{L^\infty}\leq\|u_0\|_{L^\infty}$ since
$\|u_n\|_{L^\infty}\leq\|u_0\|_{L^\infty}$ by Corollary
\ref{collcorolocal} d). We conlude that $u$ is an entropy solution of
\eqref{E} by passing to the limit in the
entropy inequality for $u_n$; cf. Definition \ref{defentsoln} c).
\end{proof}

\section{Auxiliary results}
To establish the $L^1$ contraction estimates, we will need some
auxiliary results that we derive here.

\begin{lemma}
Assume $r>0$ and that \eqref{muassumption1} holds. Let $\phi\in W^{2,1}(\R^d)$, then
\begin{equation*}
\|\mathcal{L}_r^{\mu}[\phi]\|_{L^1(\mathbb{R}^d)}\leq\frac{1}{2}\|D^2\phi\|_{L^1(\mathbb{R}^d, \,\mathbb{R}^{d\times d})}\int_{0<|z|\leq r}|z|^2\dd \mu(z) \quad\text{ for $r<1$},
\end{equation*}
\begin{equation*}
\begin{split}
\|\mathcal{L}^{\mu,r}[\phi]\|_{L^1(\mathbb{R}^d)}\leq&\, 2\|\phi\|_{L^1(\mathbb{R}^d)}\int_{|z|>r}\dd \mu(z) \quad\text{ for $r>1$},
\end{split}
\end{equation*}
and
\begin{equation*}
\|\mathcal{L}^\mu[\phi]\|_{L^1(\R^d)}\leq 2\|\phi\|_{W^{2,1}(\R^d)}\int_{\R^d\setminus\{0\}}\min\{|z|^2,1\}\dd \mu(z).
\end{equation*}
\label{L1normlowerr}
\end{lemma}

See e.g. Lemma 4.1 and Lemma 4.2 in \cite{AlCiJa12} for proofs of the above lemmas.
The main result of this section is  a "Kato inequality" or a "dual
equation" for \eqref{E}.  

\begin{proposition}
Assume \eqref{fassumption} and \eqref{Aassumption} hold. Let $u$
and $v$ be entropy sub- and supersolutions of \eqref{E} with initial
data $u_0, v_0\in L^\infty(\R^d)$ and measurable source terms $g,h$ satisfying
$\int_0^T\|g(\cdot,t)\|_{L^\infty(\mathbb{R}^d)}+\|h(\cdot,t)\|_{L^\infty(\mathbb{R}^d)}\dd t<\infty$. If either $\mathfrak{L}=\Delta$ or 
$\mathfrak{L}=\mathcal{L}^{\mu}$ and \eqref{muassumption1} holds, then
for all non-negative $\psi\in C_c^{\infty}(Q_T)$ 
\begin{equation}
\begin{split}
&\iint_{Q_T}\eta(u(x,t),v(x,t))\dell_t\psi(x,t)+q(u(x,t),v(x,t))\cdot D\psi(x,t)\dd x \dd t\\
&+\iint_{Q_T}\eta(\varphi(u(x,t)),\varphi(v(x,t)))\mathfrak{L}^*\psi(x,t)\dd x \dd t\\
&+\iint_{Q_T}\eta(g(x,t),h(x,t))\psi(x,t)\dd x \dd t\geq0,
\end{split}
\label{dualeq}
\end{equation}
where $\eta(u,v)=(u-v)^+$ and $q(u,v)=\textup{sign}(u-v)^+[f(u)-f(v)]$.
\label{dualequation}
\end{proposition}

The proof relies on the Kru\v{z}kov doubling of variables technique,
and the result is new in the non-local case. 

\begin{proof}
  If $\mathfrak{L}=\Delta$ this is a known result, see
  e.g. \cite[Theorem 3.9]{MaTo03}. The result can
  also be obtained by following the calculations of Karlsen and Risebro,
  see the proofs of Lemmas 2.3 and 2.4 and Theorem 1.1 in
  \cite{KaRi03}. Our assumptions 
and Definition \ref{defentsolnlap} ensure that equation (3.48) in \cite{KaRi03}
  holds (with $\text{Const}=0$ and
  $F(x,t,u,v)=F(u,v)=\sgn(u-v)[f(u)-f(v)]$) when the solutions $u,v$
  are in $C([0,T]; L_\textup{loc}^1(\R^d))\cap L^\infty(Q_T)$ in stead
  of $C([0,T]; L^1(\R^d))\cap L^\infty(Q_T)$. 

For $\mathfrak{L}=\mathcal{L}^{\mu}$ we follow the Proof of Theorem
3.1 in \cite{CiJa11} closely; sketching known estimates and focusing
on new ones (which are needed since $u,v\notin L^1$ anymore). We start
with the Kru\v{z}kov doubling of variables technique
\cite{Kru70,Ali07,CiJa11}. Since $u$ and $v$ are sub- and
supersolutions, we can 
take \eqref{defentsubsoln} with $u=u(x,t)$ and $k=v(y,s)$, and
\eqref{defentsupersoln} with $u=v(x,t)$ and $k=u(y,s)$. Integrate the
two inequalities over $(y,s)\in Q_T$, rename $(x,t,y,s)$ as
$(y,s,x,t)$ in the second one, and add the two inequalities. Then note
that $(v-u)^-=(u-v)^+$,
$(\varphi(v)-\varphi(u))^-=(\varphi(u)-\varphi(v))^+$, and that we can
manipulate (cf. \cite[Proof of Theorem 3.1]{CiJa11}) the integral with integrand
$\sgn(u-v)^+(\mathcal{L}^{\mu,r}[\varphi(u)]-\mathcal{L}^{\mu,r}[\varphi(v)])\phi$
to get the integrand on the form
$(\varphi(u)-\varphi(v))^+\tilde{\mathcal{L}}^{\mu^*,\,r}[\phi]$, where
\begin{equation*}
\tilde{\mathcal{L}}^{\mu^*,\,r}[\phi](x,y):=\int_{|z|>r}\phi(x+z,y+z)-\phi(x,y)\dd \mu^*(z).
\end{equation*}
Now, we let $\dd w:=\dd x \dd t \dd y \dd s$ and send $r\to0$ to find that
\begin{equation}
\begin{split}
&\iiiint_{Q_T\times Q_T}(u-v)^+(\dell_t+\dell_s)\phi\\
&\qquad\qquad\qquad+\text{sign}(u-v)^+[f(u)-f(v)]\cdot (D_x+D_y)\phi\,\dd w\\
&+\iiiint_{Q_T\times Q_T}(\varphi(u)-\varphi(v))^+\tilde{\mathcal{L}}^{\mu^*}[\phi(\cdot, t, \cdot, s)](x,y)\dd w\\
&+\iiiint_{Q_T\times Q_T}(g-h)^+\phi\, \dd w\geq0,
\end{split}
\label{lettingrtozero}
\end{equation}
where we have used that $\sgn(u-v)^+(g-h)\leq(g-h)^+$. Take 
\begin{equation*}
\phi(x,t,y,s)=\hat{\omega}_{\veps_1}\left(\frac{x-y}{2}\right)\omega_{\veps_2}\left(\frac{t-s}{2}\right)\psi\left(\frac{x+y}{2},\frac{t+s}{2}\right)
\end{equation*}
for $\veps_1,\veps_2>0$, $\psi\in C_c^{\infty}(Q_T)$ where
$\omega_\veps$ is a mollifier (see \eqref{mollifierspace}), and
$\hat{\omega}_{\veps_1}(x)$ is defined by \eqref{domega}. We
insert this test function into \eqref{lettingrtozero}, noting that
\begin{equation*}
\tilde{\mathcal{L}}^{\mu^*}[\phi(\cdot, t, \cdot, s)](x,y)=\hat{\omega}_{\veps_1}\left(\frac{x-y}{2}\right)\omega_{\veps_2}\left(\frac{t-s}{2}\right)\mathcal{L}^{\mu^*}\left[\psi\left(\cdot, \frac{t+s}{2}\right)\right]\left(\frac{x+y}{2}\right),
\end{equation*}
and then we want to take the limit as $(\veps_1,\veps_2)\to(0,0)$. 

So far the proof is quite similar to the proof of Theorem 3.1 in
\cite{CiJa11}. Taking the last limit, however, requires some
attention. Some of the arguments of \cite{CiJa11} will not hold here
since the solutions are no longer in $L^1$.

The convergence as $(\veps_1,\veps_2)\to(0,0)$ of the local terms is
well-known (cf. \cite[Proof of Theorem 6.2.3]{Daf10}), and the
convergence of the source term follows from a simple computation. So
here we give details only for the non-local term.  We need to show that
$M\to0$ for
\begin{equation*}
\begin{split}
M:=&\bigg|\iiiint_{Q_T\times Q_T}\eta(\varphi(u(x,t)),\varphi(v(y,s)))\\
&\qqquad\hat{\omega}_{\veps_1}\left(\frac{x-y}{2}\right)\omega_{\veps_2}\left(\frac{t-s}{2}\right)\mathcal{L}^{\mu^*}\left[\psi\left(\cdot, \frac{t+s}{2}\right)\right]\left(\frac{x+y}{2}\right)\dd w\\
&-\iint_{Q_T}\eta(\varphi(u(x,t)),\varphi(v(x,t)))\mathcal{L}^{\mu^*}[\psi(\cdot,t)](x)\dd x \dd t\bigg|
\end{split}
\end{equation*}
and $\eta(a,b)=(a-b)^+$.
To do that, we add and subtract
\begin{equation*}
\begin{split}
&\iiiint_{Q_T\times Q_T}\eta(\varphi(u(x,t)),\varphi(v(x,t)))\\
&\qqquad\hat{\omega}_{\veps_1}\left(\frac{x-y}{2}\right)\omega_{\veps_2}\left(\frac{t-s}{2}\right)\mathcal{L}^{\mu^*}\left[\psi\left(\cdot, \frac{t+s}{2}\right)\right]\left(\frac{x+y}{2}\right)\dd w,
\end{split}
\end{equation*}
and use that 
\begin{equation}
\iint_{Q_T}\hat{\omega}_{\veps_1}\left(\frac{x-y}{2}\right)\omega_{\veps_2}\left(\frac{t-s}{2}\right)\dd y \dd s=1,
\label{mollifiersequalsone}
\end{equation}
to get that
\begin{equation*}
\begin{split}
M\leq& \iiiint_{Q_T\times Q_T}\left|\eta(\varphi(u(x,t)),\varphi(v(y,s)))-\eta(\varphi(u(x,t)),\varphi(v(x,t)))\right|\\
&\qqquad\hat{\omega}_{\veps_1}\left(\frac{x-y}{2}\right)\omega_{\veps_2}\left(\frac{t-s}{2}\right)\mathcal{L}^{\mu^*}\left[\psi\left(\cdot, \frac{t+s}{2}\right)\right]\left(\frac{x+y}{2}\right)\dd w\\
&+\iiiint_{Q_T\times Q_T}\eta(\varphi(u(x,t)),\varphi(v(x,t)))\hat{\omega}_{\veps_1}\left(\frac{x-y}{2}\right)\omega_{\veps_2}\left(\frac{t-s}{2}\right)\\
&\qqquad\left|\mathcal{L}^{\mu^*}\left[\psi\left(\cdot,\frac{t+s}{2}\right)\right]\left(\frac{x+y}{2}\right)-\mathcal{L}^{\mu^*}[\psi(\cdot,t)](x)\right|\dd w\\
&=:M_1+M_2.
\end{split}
\end{equation*}
Since $|\eta(\varphi(u(x,t)),\varphi(v(y,s)))-\eta(\varphi(u(x,t)),\varphi(v(x,t)))|\leq |\varphi(v(x,t))-\varphi(v(y,s))|$, extensive use of adding and subtracting terms, and the triangle inequality will give
\begin{equation*}
\begin{split}
M_1\leq& \iiiint_{Q_T\times Q_T}\hat{\omega}_{\veps_1}\left(\frac{x-y}{2}\right)\omega_{\veps_2}\left(\frac{t-s}{2}\right)\\
&\qqquad\Bigg\{\left|\varphi(v(x,t))\right|\bigg|\mathcal{L}^{\mu^*}\left[\psi\left(\cdot,\frac{t+s}{2}\right)\right]\left(\frac{x+y}{2}\right)-\mathcal{L}^{\mu^*}[\psi(\cdot,t)](x)]\bigg|\\
&\qqquad+\Big|\varphi(v(x,t))\big|\Levy^{\mu^*}[\psi(\cdot,t)](x)\big|-\varphi(v(y,s))\big|\Levy^{\mu^*}[\psi(\cdot,s)](y)\big|\Big|\\ 
&\qqquad+\left|\varphi(v(y,s))\right|\bigg|\mathcal{L}^{\mu^*}\left[\psi\left(\cdot,\frac{t+s}{2}\right)\right]\left(\frac{x+y}{2}\right)-\Levy^{\mu^*}[\psi(\cdot,s)](y)\bigg|\Bigg\}\dd w.\\
\end{split}
\end{equation*}

Let us now show the convergence to zero of the term
\begin{equation*}
\begin{split}
M_2=&\iiiint_{Q_T\times Q_T}\hat{\omega}_{\veps_1}\left(\frac{x-y}{2}\right)\omega_{\veps_2}\left(\frac{t-s}{2}\right)\eta(\varphi(u(x,t)),\varphi(v(x,t)))\\
&\qqquad\left|\mathcal{L}^{\mu^*}\left[\psi\left(\cdot,\frac{t+s}{2}\right)\right]\left(\frac{x+y}{2}\right)-\mathcal{L}^{\mu^*}[\psi(\cdot,t)](x)\right|\dd w.
\end{split}
\end{equation*}
Note that $\mathcal{L^{\mu}}[\psi]\in L^1(Q_T)$ by Lemma \ref{L1normlowerr}, and that $u,v\in L^{\infty}(Q_T)$ and, hence, $\varphi(u),\varphi(v)\in L^{\infty}(Q_T)$ by \eqref{Aassumption}. By a change of variables $y-x=y'$ and $s-t=s'$, changing the order of integration, H\"older's inequality, and \eqref{mollifiersequalsone} we get
\begin{equation*}
\begin{split}
M_2\leq&\,\|\eta(\varphi(u),\varphi(v))\|_{L^\infty(Q_T)}\\
&\sup_{|y'|\leq\veps_1,\,|s'|\leq\veps_2}\left\|\mathcal{L}^{\mu^*}\left[\psi\left(\cdot,t+\frac{s'}{2}\right)\right]\left(x+\frac{y'}{2}\right)-\mathcal{L}^{\mu^*}[\psi(\cdot,t)](x)\right\|_{L^1(Q_T)},
\end{split}
\end{equation*}
which goes to zero as $(\veps_1,\veps_2)\to(0,0)$ by the continuity of
the $L^1$ translation. In a similar way, we can also show that
$M_1\to0$ and the proof is complete.
\end{proof}

In the next section we need the following corollary of Proposition \ref{dualequation}:

\begin{corollary}\label{manipulateddualeq}
Assume \eqref{fassumption}, 
\eqref{Aassumption} hold, and either $\mathfrak{L}=\Delta$ or
$\mathfrak{L}=\mathcal{L}^{\mu}$ and \eqref{muassumption1} holds. Let
$u$ and $v$ be entropy sub- and supersolutions of \eqref{E} with
initial data $u_0, v_0\in L^\infty(\R^d)$ and measurable source terms $g,h$ satisfying $\int_0^T\|g(\cdot,t)\|_{L^\infty(\mathbb{R}^d)}+\|h(\cdot,t)\|_{L^\infty(\mathbb{R}^d)}\dd t<\infty$. Let
$\psi(x,t)=\Gamma(x,t)\Theta(t)$. 
\begin{enumerate}[(a)]
\item If $0<t<T$, $0\leq\Gamma\in C_c^\infty(Q_T)$, and $0\leq\Theta\in C_c^\infty((0,T))$, then
\begin{equation}
\begin{split}
0\leq&\iint_{Q_T}(u-v)^+(x,t)\Gamma(x,t)\Theta'(t)\dd x \dd t\\
&+\iint_{Q_T}\Theta(t)(u-v)^+(x,t)\left[\dell_t \Gamma +L_f|D\Gamma|+L_\varphi\big(\mathfrak{L}^*\Gamma(x, t)\big)^+\right]\dd x \dd t\\
&+\int_0^T\Theta(t)\int_{\R^d}(g-h)^+(x,t)\Gamma(x,t)\dd x \dd t.
\end{split}
\label{mandualeq}
\end{equation}
\item If $\varphi(u)=u$ and $0\leq\Gamma\in C([0,T]; L^1(\mathbb{R}^d))\cap L^1((0,T);W^{2,1}(\mathbb{R}^d)) \cap C^{\infty}(Q_T)\cap L^\infty(Q_T)$ satisfies
\begin{equation*}
\dell_t \Gamma +L_f|D\Gamma|+\mathfrak{L}^{*}\Gamma(x, t)\leq 0 \quad
\text{in}\quad Q_T,
\end{equation*}
then
\begin{equation*}
\begin{split}
&\int_{\mathbb{R}^d}(u-v)^+(x,T)\,\Gamma(x,T)\dd x\\
&\leq \int_{\mathbb{R}^d}(u_0-v_0)^+(x)\,\Gamma(x,0)\dd x+\int_0^T\int_{\R^d}(g-h)^+(x,t)\Gamma(x,t)\dd x \dd t.
\end{split}
\end{equation*}
\item If $0\leq\Gamma\in C([0,T]; L^1(\mathbb{R}^d))\cap L^1((0,T);W^{2,1}(\mathbb{R}^d)) \cap C^{\infty}(Q_T)\cap L^\infty(Q_T)$ satisfies
\begin{equation*}
\dell_t \Gamma
+L_f|D\Gamma|+L_\varphi\big(\mathfrak{L}^{*}\Gamma(x,t)\big)^+\leq 0
\quad \text{in}\quad Q_T, 
\end{equation*}
then
\begin{equation*}
\begin{split}
&\int_{\mathbb{R}^d}(u-v)^+(x,T)\Gamma(x,T)\dd x\\
&\leq \int_{\mathbb{R}^d}(u_0-v_0)^+(x)\Gamma(x,0)\dd x+\int_0^T\int_{\R^d}(g-h)^+(x,t)\Gamma(x,t)\dd x \dd t.
\end{split}
\end{equation*}
\end{enumerate}
\end{corollary}

\begin{proof}
a) Remember that $(u-v)^+=\eta(u,v)$. The proof is a simple consequence of Equation \eqref{dualeq}, and the following easy estimates: $|q(u,v)\cdot D\Gamma|\leq|q(u,v)||D\Gamma|$, $|q(u,v)|\leq L_f\eta(u,v)$ (see \cite[p. 151]{Daf10}), and $\eta(\varphi(u),\varphi(v))\leq L_\varphi\eta(u,v)$ (by \eqref{Aassumption}) which implies that 
\begin{equation*}
\eta(\varphi(u),\varphi(v))\mathfrak{L}^{*}[\Gamma]\leq L_\varphi\eta(u,v)\big(\mathfrak{L}^*[\Gamma]\big)^+.
\end{equation*} 

\noindent b) Similar but easier than c), we omit the proof. See also
\cite{Ali07} for a proof when
$\mathfrak{L}^*=-(-\Delta)^{\frac{\alpha}{2}}$. 
\medskip

\noindent c) Since $C_c^{\infty}(Q_T)$ is dense in 
\begin{equation*}
E=\{w: w\in C([0,T];L^1(\mathbb{R}^d))\cap L^1((0,T);W^{2,1}(\mathbb{R}^d)) \text{ and } \dell_t w\in L^1(Q_T)\}
\end{equation*}
(cf. \cite[p. 159]{Ali07}), there is a sequence of functions
$\Gamma_\veps\in C_c^{\infty}(Q_T)$ such that  
\begin{equation*}
\Gamma_\veps, \dell_t \Gamma_\veps,\, |D\Gamma_\veps|,\,\mathfrak{L}^*\Gamma_\veps\to\Gamma,\dell_t\Gamma,\, |D\Gamma|,\,  \mathfrak{L}^*\Gamma \quad \text{in $L^1(Q_T)$}, 
\end{equation*}
when $\veps\to0^+$. Here we used that $\|\mathfrak{L}^*\Gamma_\veps\|_{L^1(Q_T)}\leq c\|\Gamma_\veps\|_{L^1((0,T);W^{2,1}(\R^d))}$ by the definition of $\Delta$ and by Lemma \ref{L1normlowerr}. Corollary \ref{manipulateddualeq} a) gives that Equation \eqref{mandualeq} holds with $\Gamma_\veps$ replacing $\Gamma$, and then also for $\Gamma$ by sending $\veps\to0^+$.

By \eqref{mandualeq} and the extra assumption on $\Gamma$ we see that
\begin{equation}
\begin{split}
&\iint_{Q_T}(u-v)^+(x,t)\Gamma(x,t)\Theta'(t)\dd x \dd t\\
&+\int_0^T\Theta(t)\int_{\R^d}(g-h)^+(x,t)\Gamma(x,t)\dd x \dd t\geq0.
\end{split}
\label{reddualeq}
\end{equation}
Let $0\leq\Theta\in C_c^{\infty}((0,T))$ be defined by
\begin{equation*}
\Theta(t)=\Theta_\veps(t)=\int_{-\infty}^t\omega_\veps(s-t_1)-\omega_\veps(s-t_2)\dd s,
\end{equation*}
where $0<t_1<t_2<T$. For $\veps>0$ small enough, $\Theta_\veps(t)$ is supported in $[0,T]$, and is a smooth approximation to a square pulse which is one between $t=t_1$ and $t=t_2$ and zero otherwise. By \eqref{reddualeq}, we get
\begin{equation*}
\begin{split}
&\iint_{Q_T}(u-v)^+(x,t)\Gamma(x,t)\omega_\veps(t-t_2)\dd x \dd t\\
&\leq\iint_{Q_T}(u-v)^+(x,t)\Gamma(x,t)\omega_\veps(t-t_1)\dd x \dd t\\
&\quad+\int_0^T\Theta_\veps(t)\int_{\R^d}(g-h)^+(x,t)\Gamma(x,t)\dd x \dd t.
\end{split}
\end{equation*}

Since $\eta(u,v)\in L^\infty(Q_T)$ and $\Gamma\in
C([0,T];L^1(\mathbb{R}^d))$, a direct argument, and using the
continuity of the $L^1$ translation shows the convergence of the
integrals involving $(u-v)^+\Gamma\omega_\veps$ as
$\veps\to0^+$. Moreover, since $\int_{\R^d}(g-h)^+(x,t)\Gamma(x,t)\dd
x$ is finite, Lebesgue's dominated convergence theorem will give convergence
of the integral involving $\Theta_\veps(g-h)^+\Gamma$ as
$\veps\to0^+$. Thus, we end up with 
\begin{equation*}
\begin{split}
&\int_{\mathbb{R}^d}(u-v)^+(x,t_2)\Gamma(x,t_2)\dd x\\
& \leq \int_{\mathbb{R}^d}(u-v)^+(x,t_1)\Gamma(x,t_1)\dd x\\
&\quad+\int_{t_1}^{t_2}\int_{\R^d}(g-h)^+(x,t)\Gamma(x,t)\dd x \dd t.
\end{split}
\end{equation*}

Finally, the conclusion can be obtained by letting $t_2\to T^-$ and $t_1\to 0^+$. Since $u,v\in C([0,T];L_\textup{loc}^1(\R^d))$ and $\Gamma\in C([0,T];L^1(\R^d))$, we can use Fatou's lemma on the left-hand side (the integrand is non-negative) as $t_2\to T^-$. The computations as $t_1\to 0^+$ of the first integral on the right-hand side is shown in the following:
\begin{equation*}
\begin{split}
&\|(u-v)^+(\cdot,t_1)\Gamma(\cdot,t_1)-(u-v)^+(\cdot,0)\Gamma(\cdot,0)\|_{L^1(\mathbb{R}^d)}\\
&\leq \|(u-v)^+\|_{L^{\infty}(Q_T)}\|\Gamma(\cdot,t_1)-\Gamma(\cdot,0)\|_{L^1(\mathbb{R}^d)}\\
&\quad+\|((u-v)^+(\cdot,t_1)-(u-v)^+(\cdot,0))\Gamma(\cdot,0)\|_{L^1(\mathbb{R}^d)},
\end{split}
\end{equation*}
where the first term goes to zero as $t_1\to 0^+$ since $\Gamma\in
C([0,T];L^1(\mathbb{R}^d))$. The second term, however, needs a more
refined argument. By Definition \ref{defentsolnlap} or \ref{defentsoln} a) it follows that as
$t\to 0^+$, $u(\cdot,t)\to u(\cdot,0)$ in $L_\textup{loc}^1(\R^d)$ and
hence also point-wise a.e. (along a subsequence). Moreover,
$|(u-v)^+(x,t_1)-(u-v)^+(x,0)|\Gamma(x,0)$ is dominated by
$2\|(u-v)^+\|_{L^\infty(Q_T)}\Gamma(x,0)\in L^1(\mathbb{R}^d)$. Hence,
Lebesgue's dominated convergence theorem ensures that the second term also
goes to zero when $t_1\to 0^+$.  

We conclude by using Lebesgue's dominated convergence theorem on the integral
involving $(g-h)^+\Gamma$ as $t_2\to T^-$ and $t_1\to0^+$, and by
noting that $(u-v)^+(x,0)\leq(u_0-v_0)^+(x)$ by Definition
\ref{defentsolnlap} or \ref{defentsoln} a) and b). 
\end{proof}

\section{Proof of Theorems \ref{localKcontraction} and \ref{localcontractions}}
\label{sectionlocalcon}
In previous proofs of $L^1$ contractions (see e.g. \cite{Daf10,
  Ali07}), even if it was not written in that way, the idea was
essentially to prove a result like Corollary 
\ref{manipulateddualeq} b) and then construct a suitable $\Gamma$ to
conclude. In a similar way, we will construct $\Gamma$'s for Corollary
\ref{manipulateddualeq} b) and c), and then conclude. Note that since
\eqref{viscositysoln} is fully non-linear and degenerate, this task will be much more
difficult than in \cite{Ali07} where
$\mathfrak{L}=-(-\Delta)^{\frac{\alpha}{2}}$ and $\varphi(u)=u$. 

As in \cite{Ali07}, we will build $\Gamma$ by the convolution of subsolutions of simpler problems, but first we give an auxiliary result.
 
\begin{lemma}
If $\phi\in L^1(\R^d)$ is non-negative and $f\in C_b(\mathbb{R}^d)$, then
\begin{equation*}
(\phi\ast f)^+\leq\phi\ast f^+ \text{ and}\quad |\phi\ast f|\leq\phi\ast |f|.
\end{equation*}
\label{maxlemma}
\end{lemma}
\begin{proof}
The proofs are easy and similar, so we only do one case. Since
\begin{equation*}
0\leq \int_{\mathbb{R}^d}\phi(x-y)\max\{f(y),0\}\dd y,
\label{convwithmaxgeq0}
\end{equation*}
and
\begin{equation*}
\int_{\mathbb{R}^d}\phi(x-y)f(y)\dd y\leq \int_{\mathbb{R}^d}\phi(x-y)\max\{f(y),0\}\dd y,
\label{ineqbetweenmaxconvandconvmax}
\end{equation*}
the proof is immediate.
\end{proof}

\begin{lemma}
Assume that $\mathfrak{L}=\Delta$ or $\mathfrak{L}=\mathcal{L}^\mu$ and $\eqref{muassumption1}$ holds, and assume that $0\leq\phi(x,t)\in C^\infty(Q_T)\cap C([0,T];L^1(\mathbb{R}^d))\cap L^{\infty}(Q_T)$ solves
\begin{equation}
\dell_t \phi(x,t) + L_f|D\phi(x,t)|\leq 0 \quad \text{in}\quad Q_T,
\label{subsolnphi}
\end{equation}
and define $\Gamma(x,t)=[\psi(\cdot,t)\ast\phi(\cdot,t)](x)$.
\begin{enumerate}[(a)]
\item If $0\leq\psi(x,t)\in C^\infty(Q_T)\cap C([0,T];L^1(\mathbb{R}^d))\cap L^{\infty}(Q_T)$ solves
\begin{equation*}
\dell_t \psi(x,t) +\mathfrak{L}^*\psi(x,t)\leq0 \quad \text{in}\quad
Q_T,
\end{equation*}
then $0\leq\Gamma \in C([0,T];L^1(\mathbb{R}^d))\cap C^\infty(Q_T)$,
and solves
\begin{equation*}
\dell_t\Gamma(x,t)+L_f|D \Gamma(x,t)|+\mathfrak{L}^*\Gamma(x,t)\leq0 \quad \text{in}\quad Q_T.
\end{equation*}
\item If $0\leq\psi(x,t)\in C^\infty(Q_T)\cap C([0,T];L^1(\mathbb{R}^d))\cap L^{\infty}(Q_T)$ solves
\begin{equation}
\dell_t \psi(x,t) +L_\varphi(\mathfrak{L}^*\psi(x,t))^+\leq0 \quad \text{in}\quad Q_T,
\label{subsolnpsi}
\end{equation}
then $0\leq\Gamma \in C([0,T];L^1(\mathbb{R}^d))\cap C^\infty(Q_T)$, 
and solves
\begin{equation*}
\dell_t\Gamma(x,t)+L_f|D \Gamma(x,t)|+L_\varphi\big(\mathfrak{L}^*\Gamma(x,t)\big)^+\leq0 \quad \text{in}\quad Q_T.
\end{equation*}
\end{enumerate}
\label{subsolutionofdualeq}
\end{lemma}

\begin{remark}
If $\mathfrak{L}^*=\mathfrak{L}=-(-\Delta)^{\frac{\alpha}{2}}$,
$\alpha\in(0,2]$, then Lemma \ref{subsolutionofdualeq} a) is satisfied
with $\psi(x,t)=\tilde{K}(x,\tau-t)$ for $0\leq t\leq\tau$, where
$\tilde{K}$ is defined by \eqref{heatkernel}. 
\label{laplacealpha2choosepsi}
\end{remark}

\begin{proof}
We only prove b) since a) is similar but easier. By Lemma
\ref{maxlemma} and properties of convolutions 
\begin{equation*}
\dell_t\Gamma(x,t)=\big[\dell_t\psi(\cdot,t)\ast\phi(\cdot,t)\big](x)+\big[\psi(\cdot,t)\ast\dell_t\phi(\cdot,t)\big](x),
\end{equation*}
\begin{equation*}
|D\Gamma(x,t)|\leq\big[\psi(\cdot,t)\ast|D\phi(\cdot,t)|\big](x),
\end{equation*}
and
\begin{equation*}
(\mathfrak{L}^*\Gamma(x,t))^+=\big[\phi(\cdot,t)\ast\mathfrak{L}^*\psi(\cdot,t)\big]^+(x)\leq\big[\phi(\cdot,t)\ast(\mathfrak{L}^*\psi(\cdot,t))^+\big](x).
\end{equation*}
An easy computation using \eqref{subsolnphi} and \eqref{subsolnpsi} then gives the result.
\end{proof}

To find a $\psi$ for Lemma \ref{subsolutionofdualeq}, we take the
(viscosity) solution of \eqref{viscositysoln} and mollify it. We start
by several auxiliary results and the proof of Lemma \ref{Phiprop}. 

\begin{lemma}
Assume that $\mathfrak{L}=\Delta$ or $\mathfrak{L}=\mathcal{L}^\mu$ and $\eqref{muassumption1}$ holds. If $\Phi\in C_b(Q_T)$ is a viscosity solution of \eqref{viscositysoln}, and $\rho_\delta$ is a mollifier satisfying \eqref{mollifierspacetime}, then
\begin{equation}
\Phi_\delta(x,t):=[\Phi\ast\rho_\delta](x,t)=\iint_{\mathbb{R}^d\times\mathbb{R}}\Phi(x-y,t-s)\rho_\delta(y,s)\dd y \dd s
\label{deltaconvolution}
\end{equation}
is a classical supersolution of \eqref{viscositysoln}:
\begin{equation}
\dell_t \Phi_\delta(x,t)\geq(\mathfrak{L}^*\Phi_\delta(x,t))^+.
\label{supersolnviscositysoln}
\end{equation}
\label{convolutionwithPhilemma}
\end{lemma}

\begin{remark}
As usual $\lim_{\delta\to0^+}\Phi_\delta=\Phi$ point-wise.
\end{remark}

\begin{proof}[Outline of proof]
To understand the idea behind the proof, let $\Phi(y,s)$ be a
classical solution of \eqref{viscositysoln}. Multiply the equation by
$\rho_\delta(x-y,t-s)$, integrate over $\mathbb{R}^d\times\R$
w.r.t. $(y,s)$, and use Lemma \ref{maxlemma} to conclude:
\begin{equation*}
\begin{split}
0=&\int_{\R}\int_{\mathbb{R}^d}\dell_t \Phi(y,s)\rho_\delta(x-y,t-s)\dd y\dd s\\
& -\int_{\R}\int_{\mathbb{R}^d}\left(\mathfrak{L}^*\Phi(y,s)\right)^+\rho_\delta(x-y,t-s)\dd y\dd s\\
\leq&\ \dell_t[\Phi\ast\rho_\delta](x,t)-\left(\mathfrak{L}^*[\Phi\ast\rho_\delta](x,t)\right)^+\\
=&\ \dell_t\Phi_\delta-(\mathfrak{L}^*\Phi_\delta)^+.
\end{split} 
\end{equation*}
We refer to \cite[Theorem 3.1 (a)]{BaJa07} for a proof in the case $\mathfrak{L}=\Delta$, and to \cite[Theorem 6.4]{ChJaKa08} for how to adapt this proof when $\mathfrak{L}=\mathcal{L}^\mu$.
\end{proof}

We state some well-known results for \eqref{viscositysoln}, see e.g.
\cite{CrIsLi92,JaKa05} for proofs:

\begin{lemma}
Assume that $\mathfrak{L}=\Delta$ or $\mathfrak{L}=\mathcal{L}^\mu$ and $\eqref{muassumption1}$ holds.
\begin{enumerate}[(a)]
\item If $u_0\in C_b(\R^d)$, then there exists a unique viscosity
  solution $u\in C_b(Q_T)$ of \eqref{viscositysoln}.
\smallskip
\item If $u$ and $v$ are viscosity sub- and supersolutions of
  \eqref{viscositysoln} and $u_0\leq v_0$ on $\mathbb{R}^d$, then
  $u\leq v$ in $Q_{T}$.
\smallskip
\item If $u$ is a solution of \eqref{viscositysoln} with initial data
  $u_0\in W^{1,\infty}(\mathbb{R}^d)$, then $$|u(x,t)-u(y,s)|\leq
  C(|x-y|+|t-s|^{\frac{1}{2}})\qquad\text{for}\qquad (x,t),(y,s)\in Q_{T}.$$
\item If $u$ is a classical subsolution (supersolution) of
  \eqref{viscositysoln}, then $u$ is a viscosity subsolution
  (supersolution) of \eqref{viscositysoln}. 
\end{enumerate}
\label{propofvissoln}
\end{lemma}

\begin{proof}[Proof of Lemma \ref{Phiprop}]
Since $\Phi_0(x)$ belongs to $C_c^{\infty}(\mathbb{R}^d)$ (and hence
$W^{1,\infty}(\mathbb{R}^d)$) by assumption, there exists a unique
viscosity solution $\Phi\in C_b(Q_{\tilde{T}})$ of
\eqref{viscositysoln} by Lemma \ref{propofvissoln} a). Furthermore,
since $0\leq\Phi_0(x)$, $0\leq \Phi(x,t)$ by Lemma \ref{propofvissoln}
b). 

We claim that there are $C>0$, $k>0$, $K>0$, such that for all $|\xi|=1$,
\begin{equation*}
\Phi(x,t)\leq w(x,t):=C\e^{Kt}\e^{k\xi\cdot x} \quad \text{in $Q_{\tilde{T}}$}.
\end{equation*}
If this is the case, then $\Phi(x,t)\leq C\e^{Kt}\e^{-k|x|}$ (take
$\xi=-\frac{x}{|x|}$ for $x\neq0$) and $\Phi\in L^\infty(0,\tilde{T};L^1(\R^d))$. 
Moreover, $\Phi\in C([0,\tilde{T}];L^1(\R^d))$ since by Lebesgue's dominated
convergence theorem (the integrand is 
dominated by $2C\e^{K\tilde T}\e^{-k|x|}$),
\begin{equation*}
\lim_{h\to0}\int_{\R^d}|\Phi(x,t+h)-\Phi(x,t)|\dd x=0\qquad\text{for
  all}\qquad t\in[0,\tilde T].
\end{equation*}
To complete the proof, it only remains to prove the claim.

Let $\mathfrak{L}^*=\mathcal{L}^{\mu^*}$ and assume that
\eqref{muassumption2} holds. Note that $\dell_t w=Kw$ and 
\begin{equation*}
\begin{split}
&\mathcal{L}^{\mu^*}[w(\cdot,t)](x)\\
&=\int_{|z|>0}w(x+z,t)-w(x,t)-z\cdot D w(x,t)\mathbf{1}_{|z|\leq 1}\dd \mu^*(z)\\
&=\,w(x,t)\Bigg[\int_{0<|z|\leq 1}\e^{k\xi\cdot z}-1- k\xi\cdot z\ \dd
\mu^*(z)+ \int_{|z|>1}\e^{k\xi \cdot z}-1\ \dd \mu^*(z)\Bigg]\\
\end{split}
\end{equation*}
Take $k\leq M$, where $M$ is defined in \eqref{muassumption2}. Then by Taylor's theorem and \eqref{muassumption2},
\begin{equation*}
\mathcal{L}^{\mu^*}[w(\cdot,t)](x)\leq\,C_kw(x,t),
\end{equation*}
where
\begin{equation*}
C_k:=\frac{\e^k}{2}k^2\int_{0<|z|\leq 1}|z|^2\dd \mu^*(z) + \int_{|z|>1}\e^{M|z|}\dd \mu^*(z)\in(0,\infty).
\end{equation*}
It then follows that
\begin{equation*}
\dell_t w-(\mathcal{L}^{\mu^*}[w])^+=\dell_tw + \min\{-\mathcal{L}^{\mu^*}[w],0\}\geq w(K-C_k).
\end{equation*}
We take $K$ such that $K-C_k\geq0$ in order to make $w$ a
supersolution. Now, choose $C$ such that $\Phi_0\leq
w(\cdot,0)$. Then Lemma \ref{propofvissoln} d) shows that $w$ is a
viscosity supersolution, and Lemma \ref{propofvissoln} b) ensures that
$\Phi(x,t) \leq w(x,t)$. 

When $\mathfrak{L}^*=\Delta$, the argument is similar. We take any
$k>0$ and a $C$ such that $\Phi_0\leq w(\cdot,0)$, and then we observe that 
\begin{equation*}
\dell_t w-(\Delta w)^+= w(K-k^2).
\end{equation*}
If $K-k^2\geq0$, then Lemma \ref{propofvissoln} d) and b) ensure that $\Phi(x,t) \leq w(x,t)$ as before.
\end{proof}

\begin{proposition}
Let $\Phi$ be the function given by Lemma \ref{Phiprop}, $\tilde{T}=\max\{T,L_\varphi T\}$, and $L_\varphi$ be the Lipschitz constant of $\varphi$. Then $\Phi_\delta(x,t)$ defined by \eqref{deltaconvolution} solves \eqref{supersolnviscositysoln}, satisfies
\begin{equation*}
0\leq\Phi_\delta\in C([0,\tilde{T}];L^1(\mathbb{R}^d))\cap C^{\infty}(Q_{\tilde{T}})\cap L^\infty(Q_{\tilde{T}}),
\end{equation*}
and
\begin{equation}
\|\Phi_\delta(\cdot,0)-\Phi_0\|_{L^{\infty}(\mathbb{R}^d)}\leq C\delta,
\label{approxinitialdata}
\end{equation}
where $C$ is some constant independent of $\delta>0$.
\label{Phi_deltaprop}
\end{proposition}

\begin{proof}
First note that $\Phi$, $\rho_\delta$, and hence $\Phi_\delta$, are
nonnegative, bounded, and $\rho_\delta$ and $\Phi_\delta$ are
smooth. Moreover, by Tonelli's theorem $\Phi_\delta\in
C([0,\tilde{T}]; L^1(\mathbb{R}^d))$ since 
\begin{equation*}
\int_{\mathbb{R}^d}\Phi_\delta(x,t)\dd x
=\iint_{\mathbb{R}^d\times\mathbb{R}}\rho_\delta(y,s)\int_{\mathbb{R}^d}\Phi(x-y,t-s)\dd x \dd y \dd s
\leq\max_{t\in[0,\tilde{T}]}\|\Phi(\cdot,t)\|_{L^1(\mathbb{R}^d)}.
\end{equation*}
By Lemma \ref{convolutionwithPhilemma}, $\Phi_\delta$ is a classical supersolution of \eqref{viscositysoln} and hence solves \eqref{supersolnviscositysoln}.

We use simple computations, the compact support of $\rho_\delta$, and Lemma \ref{propofvissoln} c) to obtain 
\begin{equation*}
\begin{split}
&|\Phi_\delta(x,0)-\Phi_0(x)|\\
&\leq\iint_{\mathbb{R}^d\times\mathbb{R}}\left(|\Phi(x-y,0-s)-\Phi_0(x-y)|+|\Phi_0(x-y)-\Phi_0(x)|\right)\rho_\delta(y,s)\dd y \dd s\\
&\leq\iint_{\mathbb{R}^d\times\mathbb{R}}C(|s|^{\frac{1}{2}}+|y|)\rho_\delta(y,s)\dd y \dd s\\
&\leq C\left(\sup_{s\in(0,\delta^2)}|s|^{\frac{1}{2}}+\sup_{y\in(-\delta,\delta)^d}|y|\right)\iint_{\mathbb{R}^d\times\mathbb{R}}\rho_\delta(y,s)\dd y \dd s\\
&=C\delta,
\end{split}
\end{equation*}
and hence \eqref{approxinitialdata} holds.
\end{proof}

\begin{corollary}
Let $\Phi_\delta$ be the function given by Proposition \ref{Phi_deltaprop}, $\tilde{T}=\max\{T,L_\varphi T\}$, $0<\tau<\tilde{T}$ and $0\leq t\leq \tau$, and let 
\begin{equation*}
K_\delta(x,t):=\Phi_\delta(x,L_\varphi(\tau-t)), 
\end{equation*}
where $L_\varphi$ is the Lipschitz constant of $\varphi$. Then 
\begin{equation*}
0\leq K_\delta\in C([0,\tilde{T}];L^1(\mathbb{R}^d))\cap C^{\infty}(Q_{\tilde{T}})\cap L^\infty(Q_{\tilde{T}})
\end{equation*}
solves
\begin{equation*}
\dell_t K_\delta+L_\varphi(\mathfrak{L}^*K_\delta)^+\leq0 \quad \text{in $Q_{\tilde{T}}$},
\end{equation*}
and satisfies
\begin{equation*}
\|K_\delta(\cdot,\tau)-\Phi_0\|_{L^{\infty}(\mathbb{R}^d)}\leq C\delta,
\end{equation*}
where C is a constant independent of  $\delta>0$.
\label{subsolnkernel}
\end{corollary}

To complete the collection of lemmas needed to prove Theorems
\ref{localKcontraction} 
and \ref{localcontractions}, we now show how to choose $\phi$ in Lemma
\ref{subsolutionofdualeq}. 

\begin{lemma}
Let $L_f$ be the Lipschitz constant of $f$, $0<\tau< T$, $0\leq t\leq\tau$, $R>L_fT+1$, $\tilde{\delta}>0$, $x_0\in\mathbb{R}^d$, and
\begin{equation}
\gamma_{\tilde{\delta}}(x,t):=\left[\mathbf{1}_{(-\infty, R]}\ast{\omega}_\veps\right]\left(\sqrt{\tilde{\delta}^2+|x-x_0|^2}+L_ft\right),
\label{subsolngamma}
\end{equation}
where ${\omega}_\veps$ is a mollifier (defined by \eqref{mollifierspace}). Then $\gamma_{\tilde{\delta}}\in C_c^{\infty}(Q_T)$ and
\begin{equation*}
\dell_t \gamma_{\tilde{\delta}}(x,t)+L_f|D\gamma_{\tilde{\delta}}(x,t)|\leq0.
\end{equation*}
\end{lemma}

Since $\left[\mathbf{1}_{(-\infty,R]}\ast{\omega}_\veps\right]'\leq0$ in $\R_+$, the proof is a straightforward computation.

\begin{proof}[Proof of Theorem  \ref{localcontractions}] 
Let $0<\tau< T$, $R>L_fT+1$, $x_0\in \mathbb{R}^d$, and $\veps, \delta, \tilde{\delta}>0$, and $\gamma_{\tilde{\delta}}$ be defined by \eqref{subsolngamma}. Define
\begin{equation*}
\gamma(x,t):=\lim_{\tilde\delta\to0^+}\gamma_{\tilde{\delta}}(x,t)=\left[\mathbf{1}_{(-\infty,R]}\ast{\omega}_\veps\right](|x-x_0|+L_ft)
\end{equation*} 
and
\begin{equation*}
\Gamma(x,t)=
\big[K_\delta(\cdot, t)\ast\gamma_{\tilde{\delta}}(\cdot,t)\big](x) \qquad\text{for}\qquad 0\leq t \leq \tau,
\end{equation*}
where $K_\delta$ is given by Corollary \ref{subsolnkernel}. By the properties of $K_\delta$, and since $0\leq\gamma_{\tilde{\delta}}\in C_c^{\infty}(Q_T)$, 
$$0\leq\Gamma\in C([0,\tau];L^1(\mathbb{R}^d))\cap L^1(0,\tau;W^{2,1}(\mathbb{R}^d))\cap C^\infty(Q_\tau)\cap L^\infty(Q_\tau).$$
By Lemma \ref{subsolutionofdualeq} (with $\phi=\gamma_{\tilde{\delta}}$ and $\psi=K_\delta$) and Corollary \ref{manipulateddualeq} c), it then follows that
\begin{equation*}
\begin{split}
\int_{\mathbb{R}^d}(u-v)^+(x,\tau)\,\Gamma(x,\tau)\dd x
\leq& \int_{\mathbb{R}^d}(u_0-v_0)^+(x)\, \Gamma(x,0)\dd x\\
&\quad+\int_0^\tau\int_{\R^d}(g-h)^+(x,t)\, \Gamma(x,t)\dd x\dd t,
\end{split}
\end{equation*}
or 
\begin{equation}
\begin{split}
&\int_{\mathbb{R}^d}(u-v)^+(x,\tau)\,\big[K_\delta(\cdot, \tau)\ast\gamma_{\tilde{\delta}}(\cdot,\tau)\big](x)\dd x\\
& \leq \int_{\mathbb{R}^d}(u_0-v_0)^+(x)\,\big[K_\delta(\cdot, 0)\ast\gamma_{\tilde{\delta}}(\cdot,0)\big](x)\dd x\\
&\quad+\int_0^\tau\int_{\R^d}(g-h)^+(x,t)\,\big[K_\delta(\cdot, t)\ast\gamma_{\tilde{\delta}}(\cdot,t)\big](x)\dd x\dd t.
\end{split}
\label{coarsecoarselocalcon}
\end{equation}
We use Tonelli's theorem to rewrite the right-hand side,
\begin{equation}
\begin{split}
&\int_{\mathbb{R}^d}(u_0-v_0)^+(x)\int_{\mathbb{R}^d}K_\delta(x-y,0)\gamma_{\tilde{\delta}}(y,0)\dd y \dd x\\
&=\int_{\mathbb{R}^d}\gamma_{\tilde{\delta}}(y,0)\int_{\mathbb{R}^d}(u_0-v_0)^+(x)K_\delta(x-y,0)\dd x \dd y\\
&=\int_{\mathbb{R}^d}\gamma_{\tilde{\delta}}(x,0)\,\big[K_\delta(-\cdot, 0)\ast(u_0-v_0)^+\big](x)\dd x,
\end{split}
\label{manipulationchangingorderofintegration}
\end{equation}
and similarly,
\begin{equation*}
\begin{split}
&\int_0^\tau\int_{\R^d}(g-h)^+(x,t)\,\big[K_\delta(\cdot, t)\ast\gamma_{\tilde{\delta}}(\cdot,t)\big](x)\dd x\dd t\\
&=\int_0^\tau\int_{\R^d}\gamma_{\tilde{\delta}}(x,t)\,\big[K_\delta(-\cdot,t)\ast(g(\cdot,t)-h(\cdot,t))^+\big](x)\dd x \dd t.
\end{split}
\end{equation*}

With the above manipulation in mind, we take the limit inferior of
\eqref{coarsecoarselocalcon} as $\tilde{\delta}\to0^+$ using Fatou's
lemma on the left-hand side (the integrand is nonnegative), and Lebesgue's
dominated convergence theorem on the right-hand side since the integrands
are dominated by
$2\left[\mathbf{1}_{(-\infty,2R]}\ast\omega_\veps\right](|x-x_0|+L_ft)K_\delta(-y,t)M(t)$
for $M(t)=\|u_0\|_{L^\infty(\R^d)}+\|v_0\|_{L^\infty(\R^d)}+\|g(\cdot,t)\|_{L^\infty(\R^d)}+\|h(\cdot,t)\|_{L^\infty(\R^d)}$. Thus, 
\begin{equation}
\begin{split}
&\int_{\mathbb{R}^d}(u-v)^+(x,\tau)\,\big[K_\delta(\cdot, \tau)\ast\gamma(\cdot,\tau)\big](x)\dd x\\
&\leq\int_{\mathbb{R}^d}\gamma(x,0)\,\big[K_\delta(-\cdot, 0)\ast(u_0-v_0)^+\big](x)\dd x\\
&\quad+\int_0^\tau\int_{\R^d}\gamma(x,t)\,\big[K_\delta(-\cdot,t)\ast(g(\cdot,t)-h(\cdot,t))^+\big](x)\dd x \dd t.
\end{split}
\label{coarselocalcon}
\end{equation}

By H\"older's inequality and Corollary \ref{subsolnkernel},
\begin{equation*}
\begin{split}
&\big|\big[K_\delta(\cdot,\tau)\ast\gamma(\cdot,\tau)\big](x)-\big[\Phi_0\ast\gamma(\cdot,\tau)\big](x)\big|\\
&\leq\|K_\delta(\cdot,\tau)-\Phi_0\|_{L^{\infty}(\mathbb{R}^d)}\|\gamma(\cdot,\tau)\|_{L^1(\mathbb{R}^d)}\\
&=C\delta.
\end{split}
\end{equation*}
Hence, taking the limit inferior as $\delta\to0^+$ in \eqref{coarselocalcon} using Fatou's lemma gives
\begin{equation}
\begin{split}
&\int_{\mathbb{R}^d}(u-v)^+(x,\tau)\,\big[\Phi_0\ast\gamma(\cdot,\tau)\big](x)\dd x\\
&\leq\liminf_{\delta\to0^+}\int_{\mathbb{R}^d}\gamma(x,0)\,\Big[K_\delta(-\cdot, 0)\ast(u_0-v_0)^+\Big](x)\dd x\\
&\quad+\liminf_{\delta\to0^+}\int_0^\tau\int_{\R^d}\gamma(x,t)\,\Big[K_\delta(-\cdot,t)\ast(g(\cdot,t)-h(\cdot,t))^+\Big](x)\dd x \dd t.
\end{split}
\label{takingliminfdelta}
\end{equation}

Now, let
$C_c^{\infty}(\mathbb{R}^d)\ni\Phi_0(x):=\hat{\omega}_{\tilde{\veps}}(x-x_0)$ (see \eqref{domega}). Note
that $\left[\Phi_0\ast\gamma(\cdot,\tau)\right]\geq0$ and that
$\left[\Phi_0\ast\gamma(\cdot,\tau)\right](x)=1$ when
$|x-x_0|<R-L_f\tau-\veps-\tilde{\veps}$. Hence, if
$\veps+\tilde{\veps}<1$, then 
$$\left[\Phi_0\ast\gamma(\cdot,\tau)\right](x)\geq
\mathbf{1}_{|x-x_0|\leq R-L_f\tau-1},$$
and hence we have the following lower bound for the left-hand side of \eqref{takingliminfdelta},
\begin{equation*}
\begin{split}
&\int_{\mathbb{R}^d}\mathbf{1}_{|x-x_0|\leq
  R-L_f\tau-1}(u-v)^+(x,\tau)\dd x\\
&\leq
\int_{\mathbb{R}^d}(u-v)^+(x,\tau)\,\big[\Phi_0\ast\gamma(\cdot,\tau)\big](x)\dd
x.
\end{split}
\end{equation*} 
Observe that we cannot send $\tilde{\veps}\to0^+$ here because this will violate the
inequality $w(x,0)\geq\Phi_0$ in the proof of Proposition
\ref{Phi_deltaprop}, and we would lose the $L^1$ bound on
$K_\delta$. 

Consider the first term on the right-hand side of
\eqref{takingliminfdelta}. Note that $\gamma(x,0)=\left[\mathbf{1}_{(-\infty,R]}\ast{\omega}_\veps\right](|x-x_0|)$ and $K_\delta(-\cdot,0)=\Phi_\delta(-\cdot,L_\varphi\tau)$, and define
\begin{equation*}
\begin{split}
M:=&\ \bigg|\int_{\R^d}\left[\mathbf{1}_{(-\infty,R]}\ast{\omega}_\veps\right](|x-x_0|)\,\big[\Phi_\delta(-\cdot, L_\varphi\tau)\ast(u_0-v_0)^+\big](x)\dd x\\
&\quad -\int_{\R^d}\left[\mathbf{1}_{(-\infty,R]}\ast{\omega}_\veps\right](|x-x_0|)\,\big[\Phi(-\cdot, L_\varphi\tau)\ast(u_0-v_0)^+\big](x)\dd x\bigg|\\
\leq&\int_{\R^d}\left[\mathbf{1}_{(-\infty,R]}\ast{\omega}_\veps\right](|x-x_0|)\\
&\qquad\left|\big[\Phi_\delta(-\cdot,L_\varphi\tau)\ast(u_0-v_0)^+\big](x)-\big[\Phi(-\cdot,L_\varphi\tau)\ast(u_0-v_0)^+\big](x)\right|\dd x.
\end{split}
\end{equation*}
We will show
that $M\to0$ as $\delta\to0^+$, a result which follows from Lebesgue's
dominated convergence theorem if
 $$\tilde{M}:=\left|\big[\Phi_\delta(-\cdot,L_\varphi\tau)\ast(u_0-v_0)^+\big](x)-\big[\Phi(-\cdot,L_\varphi\tau)\ast(u_0-v_0)^+\big](x)\right|\to 0$$
a.e. as $\delta\to0^+$. By the definitions of $\Phi_\delta$ and
$\rho_\delta$ (\eqref{deltaconvolution} and
\eqref{mollifierspacetime}), interchanging the order of 
integration, and H\"older's inequality, we
find that
\begin{equation*}
\begin{split}
\tilde{M}\leq&\left(\|u_0\|_{L^\infty(\R^d)}+\|v_0\|_{L^\infty(\R^d)}\right)\\
&\iint_{\R^d\times\R}\rho_\delta(\xi,s)\left\|\Phi(-\xi-\cdot,L_\varphi\tau-s)-\Phi(-\cdot,L_\varphi\tau)\right\|_{L^1(\R^d)}\dd\xi\dd s.
\end{split}
\end{equation*}
The triangle and H\"older inequalities and the compact support of
$\rho_\delta$ then gives 
\begin{align*}
\tilde{M}\leq\ &\left(\|u_0\|_{L^\infty(\R^d)}+\|v_0\|_{L^\infty(\R^d)}\right)\\
&\cdot\Bigg\{\sup_{|s|<\delta^2}\|\Phi(-\cdot,L_\varphi\tau-s)-\Phi(-\cdot,L_\varphi\tau)\|_{L^1(\R^d)}\\
&+\sup_{|\xi|<\delta}\|\Phi(-\xi-\cdot,L_\varphi\tau)-\Phi(-\cdot,L_\varphi\tau)\|_{L^1(\R^d)}
\Bigg\}.
\end{align*}
The two suprema (and hence also $\tilde M$ and $M$) converge to zero
since $\Phi\in C([0,T];L^1(\R^d))$ and by the continuity of the
$L^1$~translation, respectively.  

The second term on the right-hand side of
\eqref{takingliminfdelta} can be estimated by similar arguments (note
that $K_\delta(x,t)=\Phi_\delta(x,L_\varphi(\tau-t))$), and
when we combine all the estimates we find the following inequality:
\begin{equation*}
\begin{split}
&\int_{\mathbb{R}^d}\mathbf{1}_{|x-x_0|\leq R-L_f\tau-1}(u-v)^+(x,\tau)\dd x\\
&\leq \int_{\mathbb{R}^d}\left[\mathbf{1}_{(-\infty,R]}\ast{\omega}_\veps\right](|x-x_0|)\,\big[\Phi(-\cdot, L_\varphi\tau)\ast(u_0-v_0)^+\big](x)\dd x\\
&\quad+\int_0^\tau\int_{\mathbb{R}^d}\left[\mathbf{1}_{(-\infty,R]}\ast{\omega}_\veps\right](|x-x_0|+L_ft)\\
&\qqquad\qqquad\big[\Phi(-\cdot, L_\varphi(\tau-t))\ast(g(\cdot,t)-h(\cdot,t))^+\big](x)\dd x\dd t.
\end{split}
\end{equation*}
The integrands on the  right-hand side  are dominated by $2\mathbf{1}_{(-\infty,2R]}(|x-x_0|+L_ft)\Phi(-y,L_\varphi(\tau-t))M(t)$
where
$M(t)=\|u_0\|_{L^\infty(\R^d)}+\|v_0\|_{L^\infty(\R^d)}+\|g(\cdot,t)\|_{L^\infty(\R^d)}+\|h(\cdot,t)\|_{L^\infty(\R^d)}$,
so we can use Lebesgue's dominated convergence theorem to send $\veps\to0^+$
and obtain
\begin{equation*}
\begin{split}
&\int_{B(x_0,R-L_f\tau-1)}(u(x,\tau)-v(x,\tau))^+\dd x\\
&\leq \int_{B(x_0, R)}\big[\Phi(-\cdot,L_\varphi\tau)\ast(u_0-v_0)^+\big](x)\dd y\dd x\\
&\quad + \int_0^\tau\int_{B(x_0,R-L_ft)}\big[\Phi(-\cdot,L_\varphi(\tau-t))\ast(g(\cdot,t)-h(\cdot,t))^+\big](x)\dd x \dd t.
\end{split}
\end{equation*}
For any $M>0$, we set $R=M+1+L_f\tau$. Since $\tau\in(0,T)$ is
arbitrary, the proof of Theorem 
\ref{localcontractions} is complete.
\end{proof}

\begin{proof}[Proof of Theorem \ref{localKcontraction}]
We sketch the proof in the case when $g=0$. We proceed as in the proof
of Theorem \ref{localcontractions}, this time with the choice
$\psi(x,t)=\tilde{K}(x,\tau-t)$ for $0\leq t\leq\tau$ (see Remark
\ref{laplacealpha2choosepsi}). We obtain an inequality like
\eqref{coarsecoarselocalcon}, take the limit as
$t\to \tau^-$ in \eqref{coarsecoarselocalcon}, and find that
\begin{equation*}
\begin{split}
&\lim_{t\to\tau^-}\int_{\mathbb{R}^d}(u-v)^+(x,\tau)\,\big[\tilde{K}(\cdot, \tau-t)\ast\gamma_{\tilde{\delta}}(\cdot,\tau)\big](x)\dd x\\
& \leq \int_{\mathbb{R}^d}(u_0-v_0)^+(x)\,\big[\tilde{K}(\cdot, \tau)\ast\gamma_{\tilde{\delta}}(\cdot,0)\big](x)\dd x.
\end{split}
\end{equation*}
Following \eqref{manipulationchangingorderofintegration} (using Lemma \ref{propofheatkernel} iv)), using that $\tilde{K}$ is an approximate delta function in time, and taking the limit as $\tilde{\delta}\to 0^+$ we get 
\begin{equation*}
\begin{split}
&\int_{\R^d}\left[\mathbf{1}_{(-\infty,R]}\ast\omega_\veps\right](|x-x_0|+L_f\tau)(u(x,\tau)-v(x,\tau))^+\dd x\\
&\leq \int_{\mathbb{R}^d}\left[\mathbf{1}_{(-\infty,R]}\ast{\omega}_\veps\right]\left(|x-x_0|\right)\,\big[\tilde{K}(\cdot, \tau)\ast(u_0-v_0)^+\big](x)\dd x,
\end{split}
\end{equation*}
by Fatou's lemma, Lebesgue's dominated convergence theorem, and Lemma
\ref{propofheatkernel} iii). Taking the limit as $\veps\to0^+$ (using
Lemma \ref{propofheatkernel}~ii), Fatou's Lemma, and Lebesgue's dominated
convergence theorem) yields for any $M>0$ with $R=M+L_f\tau$ 
\begin{equation*}
\begin{split}
&\int_{B(x_0,M)}(u(x,\tau)-v(x,\tau))^+\dd x\leq \int_{B(x_0, M+L_f\tau)}\big[\tilde{K}(\cdot,\tau)\ast(u_0-v_0)^+\big](x)\dd x.
\end{split}
\end{equation*}
\end{proof}

\section*{Acknowledgments}
We would like to thank Jerome Droniou for putting us on the track to
the right solution, Harald Hanche-Olsen for the many helpful
discussions on technical issues, and Boris Andreianov for pointing out an
incorrect claim in the first version and clarifying the relations to
the literature. We would also like to thank the referees for
many good questions, remarks, and suggestions which has helped us improve the
presentation.

\end{document}